\documentclass{amsart}
\usepackage{shortsalch, xy}
\xyoption{all}
\title{Ravenel's algebraic extensions of the sphere spectrum do not exist}
\date{June 2015.}
\begin{document}
\begin{abstract}
In this paper we prove a topological nonrealizability theorem: certain classes of graded $BP_*$-modules are shown to never occur as the $BP$-homology of a spectrum. Many of these $BP_*$-modules admit the structure of $BP_*BP$-comodules, meaning that their topological nonrealizability does not follow from earlier results like Landweber's filtration theorem. As a consequence we solve Ravenel's 1983 problem on the existence of ``algebraic extensions of the sphere spectrum'': algebraic extensions of the sphere spectrum do not exist, except in trivial cases.
\end{abstract}
\maketitle
\tableofcontents

\section{Introduction.}

This paper 
solves an open problem posed by D. Ravenel in 1983 on topological realizability of the classifying rings of formal groups with complex multiplication, i.e., ``formal modules.''\footnote{In this paper, all formal groups and formal modules are assumed to be {\em one-dimensional}.} 
Given a commutative ring $A$, a {\em formal $A$-module} is a formal group law $F$ equipped with a ring map $A \stackrel{\rho}{\longrightarrow} \End(F)$ such that $\rho(a)(X) \equiv aX \mod X^2$. There exists a classifying ring $L^A$ of formal $A$-modules, and if $A$ is a discrete valuation ring,
there exists a classifying ring $V^A$ of $A$-typical formal $A$-modules. See~\cite{cmah1} and~\cite{cmah2} for general properties and explicit computations of the rings $L^A$, and~\cite{MR2987372} and~\cite{MR745362} for $V^A$.
The ring $L^A$ is naturally a commutative graded algebra over Lazard's classifying ring $L$ of formal group laws, i.e., the coefficient ring $MU_*$ of complex bordism, by the ring map $L \rightarrow L^A$ classifying the underlying formal group law of the universal formal $A$-module.
For the same reasons, the ring $V^A$ is naturally a commutative graded algebra over the classifying ring $V$ of $p$-typical formal group laws, i.e., the coefficient ring $BP_*$ of $p$-primary Brown-Peterson homology. See~\cite{MR0253350} for this relationship between $L$ and $MU_*$ and between $V$ and $BP_*$.

Now here is a natural question:
\begin{question} \label{realization question} 
\begin{description}
\item[Global realization] Given the ring of integers $A$ in a finite extension $K/\mathbb{Q}$ of degree $>1$, does there exist a spectrum $S^A$ such that $MU_*(S^A) \cong L^A$?
\item[Local realization] Given the ring of integers $A$ in a finite extension $K/\mathbb{Q}_p$ of degree $>1$, does there exist a spectrum $S^A$ such that $MU_*(S^A) \cong L^A$?
\item[$p$-typical local realization] Given the ring of integers $A$ in a finite extension $K/\mathbb{Q}_p$ of degree $>1$, does there exist a spectrum $S^A$ such that $BP_*(S^A) \cong V^A$?
\end{description}
\end{question}
Ravenel asked the ``Local realization'' part of Question~\ref{realization question} in~\cite{MR745362}. The ``Global realization'' and ``$p$-typical local realization'' parts of Question~\ref{realization question} are natural extensions of Ravenel's original question.
(The reason for the restriction to degree $>1$ is that the degree $1$ cases, $K =\mathbb{Q}$ and $K = \mathbb{Q}_p$, are known classically to have positive answers: the relevant spectra are simply the sphere spectrum and its $p$-completion, respectively. See~\cite{MR0253350}.)

Question~\ref{realization question} is a reasonable one because, on the one hand, $L^A$ and $V^A$ are very interesting (number-theoretically quite meaningful) $MU_*$- and $BP_*$-algebras, so one wants to know if they occur as the complex bordism or Brown-Peterson homology of some spectrum;
and on the other hand, Ravenel remarks 
at the end of~\cite{MR745362} that, if such a spectrum $S^A$ exists,
then the cohomology $\Ext_{(L^A,L^AB)}^{*,*}(L^A, L^A)$ of the classifying Hopf algebroid of one-dimensional formal $A$-modules
(in other words, the flat cohomology of the moduli stack of formal $A$-modules) is the $E_2$-term of the Adams-Novikov spectral sequence
computing the stable homotopy groups $\pi_*(S^A)$. It is Ravenel's suggestion that the spectra $S^A$ ought to be thought of as ``algebraic extensions of the sphere spectrum.''

The ``$p$-typical local realization'' part of Question~\ref{realization question}  was taken up by A. Pearlman in his 1986 Ph.D. thesis~\cite{pearlmanthesis}. Pearlman was able to show that, for a 
totally ramified finite extension $K/\mathbb{Q}_p$ with ring of integers $A$, there exists a $BP$-module spectrum $Y$ and an isomorphism $\pi_*(Y) \cong V^A$ of graded $BP_*$-modules; however, Pearlman was not able to resolve the $p$-typical local realization problem, i.e., he did not resolve the question of whether $Y$ splits as $BP\smash S^A$ for some spectrum $S^A$.  
Pearlman also made some computations of the Morava $K$-theories of such a spectrum $S^A$, assuming its existence;
and in the 2007 paper~\cite{MR2262886}, T. Lawson proved that, when $K/\mathbb{Q}_p$ is a finite field extension of degree $>1$ and with ring of integers $A$, there does not exist an $A_p$-ring spectrum (i.e., associative-up-to-$p$th-order-homotopies) $X$ such that $MU_*(X)\cong L^A$ as graded rings.
This is a negative answer to a {\em multiplicative analogue} of the local case of Question~\ref{realization question}.
Lawson's multiplicative argument uses a Kudo-Araki-Dyer-Lashof operation in an essential way (in the case of an $A_p$-ring spectrum, by identifying it with a $p$-fold Massey product),
and so it cannot be reduced to an argument about spectra rather than ring spectra.  Lawson also showed, by a different argument, that there does not
exist a spectrum (no ring structure assumed) $X$ such that $MU_*(X) \cong L^{\hat{\mathbb{Z}}_2[i]}$, where $\hat{\mathbb{Z}}_2[i]$ is the 
ring of integers in the $2$-adic rationals with a primitive fourth root of unity adjoined, but Lawson remarks that his argument
does not seem to generalize (to odd primes, for example).
So Ravenel's original question, Question~\ref{realization question}, has remained open except
in the case $A = \hat{\mathbb{Z}}_2[i]$, where it has a negative answer.

In the present paper we settle all cases, global\footnote{In the global case---but not in the $p$-typical local case or Ravenel's original case, the local case---there is a small caveat: we assume that the extension $K/\mathbb{Q}$ is either cyclic of prime power order, or Galois. This is so that we can find a prime number $p$ which neither ramifies nor splits completely in $\mathcal{O}_K$, and such that every prime of $\mathcal{O}_K$ over $p$ has the same degree.
If one knows, for a given extension $K/\mathbb{Q}$, that such a prime exists, then it is unnecessary to assume that $K/\mathbb{Q}$ is cyclic of prime power order or Galois in order to get a negative answer to the global case of Question~\ref{realization question}.} 
and local and $p$-typical local, of Question~\ref{realization question}. The answer turns out to be: 
\begin{answer}
No.
\end{answer}
In the $p$-typical local case, the method of proof is an analysis of what kinds of $BP_*$-modules can actually occur as the $BP$-homology of a spectrum.
The arguments use some topology, and are not purely algebraic, unlike earlier topological nonrealizability results on $BP_*$-modules. Landweber's filtration theorem~\cite{MR0423332}, for example,
gives a condition that a $BP_*$-module must satisfy in order to be a $BP_*BP$-comodule and hence possibly the $BP$-homology of a spectrum,
but the $BP_*$-module $V^A$ of Question~\ref{realization question} actually {\em is} a $BP_*BP$-comodule, so purely algebraic
(non-topological) criteria like Landweber's filtration theorem do not shed any light on Question~\ref{realization question}. 

The local and global cases then follows from the $p$-typical local case, but not in a totally trivial way: the proof I give for the global case, which is the simplest proof I know, uses the fact that in every finite extension of $\mathbb{Q}$ there are infinitely many primes which do not split completely, which indicates that there is at least a little bit of nontrivial (although well-known) number theory involved.

The proof I give in the $p$-typical local case is really an answer to a more general question, which is as follows: we have the ring map $BP_*\rightarrow V^A$ classifying the underlying $p$-typical formal group law of the universal $A$-typical formal
$A$-module, and via this ring map, every $V^A$-module is also a $BP_*$-module. 
So one can ask what spectra $X$ have the property that $BP_*(X)$ is a $V^A$-module, not just a $BP_*$-module, where again $A$ is the ring of integers
in a finite extension $K/\mathbb{Q}_p$ of degree $>1$.
The answer to this question is given in Theorem~\ref{unramified dissonance thm}: if $K/\mathbb{Q}_p$ is not totally ramified, then the only spectra $X$ such that $BP_*(X)$ is a $V^A$-module are dissonant spectra (``dissonant'' in the sense of~\cite{MR737778}; I give a brief recap of what this means in Remark~\ref{dissonance remark}). Such spectra do indeed exist, for example, $X = H\mathbb{F}_p$; see the comments at the beginning of~\cref{non-tot-ram section}.
Consequently the only $V^A$-modules which occur as the $BP$-homology of a spectrum are $V^A$-modules
on which all of the Hazewinkel generators $v_n$ of $BP_*$ act nilpotently, i.e., for each element $m\in M$, there exists $i\in\mathbb{N}$ such that $v_n^im = 0$. If $K/\mathbb{Q}_p$ is totally ramified, then the situation is a little more forgiving: in Theorem~\ref{main thm on tot ram nonexistence} we find that the only spectra $X$ such that $BP_*(X)$ is a $V^A$-module are extensions of rational spectra by dissonant spectra. 

In the case that $X$ {\em is} a spectrum whose $BP$-homology admits the structure of a $V^A$-module,
then, at least when $K/\mathbb{Q}_p$ is totally ramified, there exists a change-of-rings isomorphism
describing the $E_2$-term of the Adams-Novikov spectral sequence computing $\pi_*(X)_{(p)}$ in terms of the classifying Hopf algebroid
$(V^A,V^AT)$ of formal $A$-modules: 
\begin{equation}\label{change-of-rings iso 100} \Ext_{(BP_*,BP_*BP)}^{*,*}(BP_*,BP_*(X)) \cong \Ext_{(V^A,V^AT)}^{*,*}(V^A,BP_*(X)),\end{equation}
as long as $BP_*(X)$ is flat over $BP_*$.
This is proven in~\cite{cmah1}. In the case $BP_*(X) = V^A$, the existence of such a change-of-rings isomorphism 
seems to have been known to Ravenel (who says it follows from ``standard arguments'' at the end of~\cite{MR745362}), 
and to Pearlman (who gives a sketch of a proof in~\cite{pearlmanthesis}). However, in the present paper it is proven that there is 
no spectrum $X$ such that $BP_*(X) \cong V^A$ unless $K = \mathbb{Q}_p$, rendering that special case moot. 
But actually writing out the ``standard arguments,'' as is done in~\cite{cmah1}, leads to the more general isomorphism~\ref{change-of-rings iso 100},
which does indeed apply in some cases, e.g. $X = H\mathbb{Q}_p^{\vee d}$ where $d = [K : \mathbb{Q}_p]$, when $K/\mathbb{Q}_p$ is totally ramified.

The conclusion here 
(with the understanding that, in the global case, we assume that the extension $K/\mathbb{Q}$ is either cyclic of prime power order, or Galois) is Corollary~\ref{main conclusion cor}: 
\begin{conclusion} Except in the trivial cases where $A = \mathbb{Z}$ or $A = \hat{\mathbb{Z}}_p$, Ravenel's algebraic extensions of the sphere spectrum do not exist.
\end{conclusion}

This paper ends with a short appendix on the most basic ideas from local cohomology, for the reader who does not already know them. These ideas are easy, relevant, and useful, and I make use of them starting with the proof of Theorem~\ref{main thm on tot ram nonexistence}.

\begin{remark}
I want to take a moment, however, to defend the use of formal modules in homotopy theory: despite the nonexistence of Ravenel's algebraic spheres, formal modules are very useful for a homotopy theorist! 
There exist families of ``height-shifting'' spectral sequences which, generally speaking, take as input
a height $d$ formal $A$-module computation  and output a height $n$ formal group law computation, where $A$ is the ring of integers in a field extension $K/\mathbb{Q}_p$, and $[K: \mathbb{Q}_p]\cdot d = n$.
There also exist families of spectral sequences which take as input
a height $d$ formal group computation  and output a height $d$ formal $A$-module computation. 
 See~\cite{height4} for this material, where both methods are used in order to compute the cohomology of the height $4$ Morava stabilizer group at large primes, using the cohomology of the height $2$ Morava stabilizer group as input.
In fact, 
it is the most {\em useful} properties of formal modules---that the underlying formal group of an $A$-height $n$ formal $A$-module
has $p$-height $n\cdot [K: \mathbb{Q}_p]$, i.e., the map of moduli stacks $\mathcal{M}_{fmA} \rightarrow \mathcal{M}_{fg}$ sends the height $n$ stratum to the height $n\cdot [K: \mathbb{Q}_p]$ stratum, allowing ``height-shifting'' techniques to be used in the computation of the flat cohomology of $\mathcal{M}_{fmA}$ and its various substacks---which are responsible for the map $MU_* \cong L \rightarrow L^A$ having the algebraic properties which result in the nonexistence results proven in the present paper. So while we do not have the certain kind of topological usefulness of formal modules (as in Question~\ref{realization question}) which Ravenel wanted, we instead have other kinds of topological usefulness of formal modules.
\end{remark}

I am grateful to E. Friedman for suggesting to me that it is possible to remove the assumption that $K/\mathbb{Q}$ is Galois from my original proof that ``global'' algebraic extensions of the sphere do not exist. I am especially grateful to D. Ravenel for teaching me a great deal about homotopy theory and formal groups over the years. 

\begin{conventions}
\begin{enumerate}
\item Whenever I refer to $v_n$ or $BP$ or $E(n)$, I am implicitly choosing a prime number $p$, and referring to the $p$-primary $v_n$ or $BP$ or $E(n)$. Unless otherwise specified, $v_1, v_2, \dots$ will be Hazewinkel's generators for $BP_*$, rather than Araki's.
\item When $(A,\Gamma)$ is a graded Hopf algebroid, $M,N$ are graded $\Gamma$-comodules, and I write $\Ext^{s,t}_{(A,\Gamma)}(M,N)$, I am referring to the $s$th relative right-derived functor of $\hom_{\gr\Comod(\Gamma)}(\Sigma^tM, -)$ applied to $N$; here $\gr\Comod(\Gamma)$ is the category of graded $\Gamma$-comodules, and the allowable class (in the sense of relative homological algebra, as in~\cite{MR0178036} or chapter~IX of~\cite{MR1344215}) in question is the class of extensions of graded $\Gamma$-comodules whose underlying extensions of graded $A$-modules are split. Appendix~1 of~\cite{MR860042} is the standard reference for $\Ext$ in categories of comodules relative to this allowable class.
\item When $R$ is a commutative ring, $r\in R$, and $M$ is an $R$-module, I will write $r^{-1}M$ for the localization of $M$ inverting the principal ideal $(r)$, i.e., $r^{-1}M \cong M \otimes_R R[\frac{1}{r}]$.
\end{enumerate}
\end{conventions}

\section{The map $BP_*\rightarrow V^A$.}

Recall (from~\cite{MR2987372} or from~\cite{MR745362}) that M. Hazewinkel proved that, if $A$ is a discrete valuation ring with finite residue field, then the classifying ring $V^A$ of $A$-typical formal $A$-modules is isomorphic to $A[v_1^A, v_2^A, \dots]$, a polynomial algebra over $A$ on countably infinitely many generators.
Now suppose we have an extension $L/K$ of $p$-adic number fields, 
and suppose that $A,B$
are the rings of integers in $K,L$, respectively. Then 
there are induced maps of classifying rings 
\begin{align*} 
V^{A} & \stackrel{\gamma}{\longrightarrow} V^{B} \mbox{\ \ \ and} \\
V^{A}\otimes_A B & \stackrel{\gamma^{\sharp}}{\longrightarrow} V^{B}\end{align*}
given by classifying the ${{A}}$-typical formal ${{A}}$-module law underlying the universal ${{B}}$-typical
formal ${{B}}$-module law on $V^{{B}}$. 
One naturally wants to know whether $\gamma^{\sharp}$ is injective, surjective, or has other recognizable properties
(this is necessary in order to prove the topological nonrealizability theorems in~\cref{non-tot-ram section} and~\cref{tot ram case}, for example).

In this section, the two main results are:
\begin{itemize}
\item
Corollary~\ref{structure of vb over va}: if $L/K$ is unramified,
then $\gamma^{\sharp}$ is surjective, and we compute $V^B$ 
as a quotient of $V^A\otimes_A B$. (That particular computation is not really new; the case $K = \mathbb{Q}_p$ appeared as Corollary~2.7 in~\cite{MR745362}, and the proof for general $K$ is essentially the same.)
\item 
Theorem~\ref{comparison map is a rational iso}:
if $L/K$ is totally ramified, then $\gamma^{\sharp}$ is a rational isomorphism. 
That is, $\gamma^{\sharp}$ is injective, and $\gamma^{\sharp}$ is an isomorphism after inverting the prime $p$. (This result, on the other hand, is new.)
\end{itemize}

\begin{remark}\label{remark on log coefficients}
In this section we make heavy use of Hazewinkel's generators $v_n^A$ for $V^A$. They are defined as follows: let 
$A$ be a discrete valuation ring of characteristic zero and with finite residue field, let
$q$ denote the cardinality of the residue field of $A$, let $\pi$ denote a uniformizer for $A$, and let
\[ \log_{F^A}(X) = \sum_{n\geq 1} \ell_n^A X^{q^n}\]
be the logarithm of the universal $A$-typical formal $A$-module $F^A$. Then the logarithm coefficients $\ell_n^A \in \mathbb{Q}\otimes_{\mathbb{Z}} A$
satisfy the equations
\begin{align}
\nonumber \ell_0^A &= 1,\\
\label{hazewinkel formula}  \pi\ell_n^A &= \sum_{i=0}^{n-1} \ell^A_i(v_{n-i}^A)^{q^i} 
\end{align}
for some set of parameters $v_0^A = \pi, v_1^A, v_2^A, \dots$,
and these are the Hazewinkel generators for $V^A$. This material appears in~21.5.4 of~\cite{MR506881}.
It follows from Hazewinkel's construction of the universal $A$-typical formal $A$-module, using Hazewinkel's functional equation lemma, that
the map $V^A\stackrel{\gamma}{\longrightarrow} V^B$ is ``natural in the logarithm,'' that is, $\gamma$ sends the coefficient of the degree $n$ term in $\log_{F_A}$ to the coefficient of the degree $n$ term of $\log_{F_B}$, for all $n$. (It is really $\mathbb{Q}\otimes_{\mathbb{Z}} \gamma: \mathbb{Q}\otimes_{\mathbb{Z}} V^A \rightarrow \mathbb{Q}\otimes_{\mathbb{Z}} V^B$ that can be evaluated on the coefficients in the power series $\log_{F_A}$, not $\gamma$ itself, since the logarithm coefficients involve denominators; when convenient I will write $\gamma(\ell_n^A)$ as shorthand for $(\mathbb{Q}\otimes_{\mathbb{Z}}\gamma)(\ell_n^A)$.)
Basically all of the arguments in the rest of this section consist in then solving the resulting equations to see what $\gamma$ does to the Hazewinkel generators $v_0^A,v_1^A,v_2^A, \dots$.
\end{remark}

The case $K = \mathbb{Q}_p$ of Proposition~\ref{computationofunramifiedgamma} 
also appears as Corollary~2.7 in~\cite{MR745362}.
\begin{prop}\label{computationofunramifiedgamma} 
Let $K,L$ be $p$-adic number fields, and let $L/K$ be 
unramified of degree $f$. Let $A,B$ denote the rings of integers of $K$ and $L$, respectively.
Using the Hazewinkel generators for $V^{{A}}$ and $V^{{B}}$, the
map $V^{{A}}\stackrel{\gamma}{\longrightarrow}V^{{B}}$ is determined by
\[ \gamma(v^{{A}}_i) = \left\{ \begin{array}{ll} v^{{B}}_{i/f} & {\rm if\ } f\mid i\\
0 & {\rm if\ } f\nmid i\end{array} \right. .\]\end{prop}
\begin{proof} 
Let $\pi$ be a uniformizer for $A$.
Since $L/K$ is unramified, $\pi$ is also a uniformizer for $B$.
Let $q$ be the cardinality of the residue field of $A$.
We immediately have that $v_0^A = \pi = v_0^B$, so $\gamma(v_0^A) = v_0^B$.

Now, since the residue field of $B$ is $\mathbb{F}_{q^f}$, the
logarithm $\log_{F^B}(X) = \sum_{n\geq 1} \ell_{n}^B X^{q^{fn}}$
has no nonzero terms $X^{q^i}$ for $f\nmid i$.
Hence $\gamma(\ell^{A}_i) = 0$ if $f\nmid i$.
So we have that 
$\gamma(v_i^{{A}}) = 0$ for $0<i<f$, and 
\begin{eqnarray*} \frac{1}{\pi}v_1^{{B}} & = & \ell_1^{{B}}\\
& = & \gamma(\ell_f^{{A}}) \\
& = & \gamma\left(\frac{1}{\pi}\sum_{i=0}^{f-1}\ell_i^{{A}}(v^{{A}}_{f-i})^{q^i}\right)\\
& = & \frac{1}{\pi}\gamma(v^{{A}}_f),\end{eqnarray*}
so $\gamma(v^{{A}}_f) = v_1^{{B}}$.

We proceed by induction. Suppose that $\gamma(v^{{A}}_i) = 0$ if $f\nmid i$ and $0<i<hf$, and that
$\gamma(v^{{A}}_i) = v^{{B}}_{i/f}$ if $f\mid i$ and $i< hf$.
Then 
\begin{align*}
 v_h^{{B}}+\sum_{i=1}^{h-1}\ell^{{B}}_i(v_{h-i}^{{B}})^{q^{if}}
  &= \ell^{{B}}_h \\
  &= \gamma(\ell^{{A}}_{hf}) \\
  &= \gamma\left(\sum_{i=0}^{hf-1} \ell^{{A}}_i(v_{hf-i}^{{A}})^{q^i}\right)\\
  &=  \sum_{i=0}^{h-1}\ell^{{B}}_i\gamma(v_{(h-i)f}^{{A}})^{q^{if}}\\ 
  &= \gamma(v_{hf}^{{A}})+ \sum_{i=1}^{h-1}\ell^{{B}}_i (v_{h-i}^{{B}})^{q^{if}},
\end{align*}
so $\gamma(v^{{A}}_{hf}) = v^{{B}}_h$.

Now suppose that $0<j<f$ and $\gamma(v^{{A}}_i) = 0$ if $f\nmid i$ and $i<hf+j$, and suppose that
$\gamma(v^{{A}}_i) = v^{{B}}_{i/f}$ if $f\mid i$ and $i< hf$. We want to show that $\gamma(v^{{A}}_{hf+j}) = 0$.
We see that \begin{eqnarray*} 0 & = & \gamma(\ell^{{A}}_{hf+j}) \\ & = & \gamma\left(\sum_{i=0}^{hf+j-1}\ell^{{A}}_i(v^{{A}}_{hf+j-i})^{q^i}\right)\\
& = & \sum_{i=0}^h\ell^{{B}}_i\gamma(v^{{A}}_{(h-i)f+j})^{q^{if}}\\
& = & \gamma(v^{{A}}_{hf+j}).\end{eqnarray*}
This concludes the induction.\end{proof}
\begin{corollary} \label{structure of vb over va} 
Let $K,L$ be $p$-adic number fields with rings of integers $A,B$, respectively,
and let $L/K$ be unramified of degree $f$. Then 
\[ V^{{B}}\cong 
(V^{{A}}\otimes_{{A}}{{B}})/\left(\{v^{{A}}_i: f\nmid i\}\right).\]
\end{corollary}
When $L/K$ is ramified, the map $\gamma$ is much more complicated. 
One can always use the above formulas to solve the
equation $\gamma(\ell_i^A) = \ell_i^B$
to get a formula for $\gamma(v_i^A)$, as long as one knows 
$\gamma(v_j^A)$ for all $j<i$; but it seems that there is no simple closed form
for $\gamma(v_i^A)$ for general $i$. Here are formulas for low $i$:
\begin{prop} \label{low degree gamma computation}
Let $K,L$ be 
$p$-adic number fields with rings of integers $A$ and $B$,
and let $L/K$ be 
totally ramified.
Let $A$ (and consequently also $B$) have residue field $\mathbb{F}_q$.
Let $\pi_A,\pi_B$ be uniformizers for $A,B$, respectively.
Let $\gamma$ be the ring homomorphism 
$V^A\stackrel{\gamma}{\longrightarrow} V^B$ classifying the underlying formal $A$-module law of the 
universal formal $B$-module law. 
Then
\begin{eqnarray*}
\gamma(v_1^A) & = & \frac{\pi_A}{\pi_B}v_1^B, \\
\gamma(v_2^A) & = & \frac{\pi_A}{\pi_B}v_2^B + 
\left(\frac{\pi_A}{\pi_B^2}-\frac{\pi_A^q}{\pi_B^{q+1}}\right)(v_1^B)^{q+1}.
\end{eqnarray*}
\end{prop}
\begin{proof} 
These formulas follow immediately from setting $\gamma(\ell_i^A) = \ell_i^B$ and solving, using formula~\ref{hazewinkel formula}.\end{proof}

\begin{prop}\label{rational surjectivity}
 Let $K,L$ be $p$-adic number fields with rings of
integers $A,B$, respectively, and let $L/K$ be totally ramified.
Write $\gamma^{\sharp}$ for the ring map $V^A\otimes_A B\rightarrow V^B$
classifying the underlying $A$-typical formal $A$-module of the universal
$B$-typical formal $B$-module.
Then $\gamma^{\sharp}$ is rationally surjective, that is, the map
\[ V^{{A}}\otimes_A B \otimes_{\mathbb{Z}}\mathbb{Q}
\stackrel{\gamma^{\sharp}\otimes_{\mathbb{Z}}\mathbb{Q}}{\longrightarrow}
V^{{B}}\otimes_{\mathbb{Z}}\mathbb{Q}\] 
is surjective.\end{prop}
\begin{proof} 
Let $\pi_A$ be a uniformizer for $A$, and let $\pi_B$ be a uniformizer for $B$.
Given any $v^{{B}}_i$, we want to produce some element of $V^{{A}}\otimes_A B\otimes_{\mathbb{Z}}\mathbb{Q}\cong V^A \otimes_A L$ which maps to it. We 
begin with $i=1$. 
By Proposition~\ref{low degree gamma computation}, we have $(\gamma^{\sharp}\otimes_{\mathbb{Z}}\mathbb{Q})(\frac{\pi_B}{ \pi_A}v_1^{{A}}) = v_1^{{B}}$. 

Now we proceed by induction. Suppose that we have shown that there is an element in $(v_1^{{A}},\dots ,v_{j-1}^{{A}})\subseteq V^{{A}}\otimes_{\mathbb{A}}\mathbb{L}$ which 
maps via $(\gamma^{\sharp}\otimes_{\mathbb{Z}}\mathbb{Q})$ to $v_{j-1}^{{B}}$. We have
$(\gamma^{\sharp}\otimes_{\mathbb{Z}}\mathbb{Q})(\ell_j^A) = \ell_j^B$, i.e.,
\begin{equation}\label{hazewinkel log coeff eq} (\gamma^{\sharp}\otimes_{\mathbb{Z}}\mathbb{Q})\left(\pi_A^{-1}\sum_{i=0}^{j-1}\ell_i^{{A}}(v_{j-i}^{{A}})^{q^i}\right) = \pi_B^{-1} \sum_{i=0}^{j-1}\ell_i^{{B}}(v_{j-i}^{{B}})^{q^i}\end{equation}
and hence
\begin{eqnarray}\nonumber(\gamma^{\sharp}\otimes_{\mathbb{Z}}\mathbb{Q})(\pi_A^{-1}v_j^{{A}}) & = & 
\pi_B^{-1}\sum_{i=0}^{j-1}\ell_i^{{B}}(v_{j-i}^{{B}})^{q^i}-(\gamma^{\sharp}\otimes_{\mathbb{Z}}\mathbb{Q})(\pi_A^{-1}\sum_{i=1}^{j-1}\ell_i^{{A}}(v_{j-i}^{{A}})^{q^i})\\
\label{totramgammamodideal} & \equiv & \pi_B^{-1}v_j^{{B}}\mod (v_1^{{B}},v_2^{{B}},\dots ,v_{j-1}^{{B}})\subseteq V^B\otimes_{\mathbb{Z}}\mathbb{Q}.\end{eqnarray}
So $(\gamma^{\sharp}\otimes_{\mathbb{Z}}\mathbb{Q})(\frac{\pi_B}{\pi_A}v_j^{{A}}) \equiv v_j^{{B}}$ modulo terms hit by elements in the ideal generated by Hazewinkel generators of lower 
degree.
This completes the induction.\end{proof}

For the proof of the next proposition we will use a monomial ordering on $V^A$. 
\begin{definition}\label{vaordering} We put the following ordering on the Hazewinkel generators of $V^A$:
\begin{eqnarray*} v_i^A\leq v_j^A & \mbox{iff}& i\leq j\end{eqnarray*}
and we put the lexicographic order on the monomials of $V^A$.
This ordering on the monomials in $V^A$ is a total ordering (see e.g. \cite{MR2122859} for these kinds of methods).

We extend this ordering to an ordering on all of $V^A$ by letting $x\leq y$ if and only if the maximal (in the given total ordering on monomials) monomial
in $x$ is less than or equal to the maximal monomial in $y$.
\end{definition}
So, for example, $v_3^A$ is ``greater than'' every monomial in the generators $v_1^A$ and $v_2^A$;
the monomial $(v_2^A)^2$ is ``greater than'' $(v_1^A)^nv_2^A$ for all $n$; and so on.

\begin{lemma}\label{gamma preserves ordering}
Let $K,L$ be finite extensions of $\mathbb{Q}_p$ with rings of integers $A,B$, respectively, and let $L/K$ be a totally ramified extension.
Let $x,y$ be monomials in $V^A$. Then $x\leq y$ implies that $\gamma^{\sharp}(x)\leq \gamma^{\sharp}(y)$.
\end{lemma}
\begin{proof}
Let $q$ be the cardinality of the residue field of $A$ (hence also $B$), and let $\pi_A,\pi_B$ be uniformizers of $A,B$, respectively.
Equation~\ref{hazewinkel log coeff eq} implies that
\begin{equation}\label{hazewinkel log coeff eq 2} \gamma(v_n^A) = \frac{\pi_A}{\pi_B} v_n^B +\sum_{i=1}^{n-1} \ell_i^B \left( \frac{\pi_A}{\pi_B} (v_{n-i}^B)^{q^i} - \gamma(v_{n-i}^A)^{q^i}\right) .\end{equation}
Suppose that $x\leq y$.
Let $n_x$ be the largest integer $n$ such that $v_n^A$ divides $x$, and similar for $n_y$.
If $n_x < n_y$, then $\gamma^{\sharp}(y)$ has monomials divisible by $v_{n_y}^B$ but 
there is no monomial of $\gamma^{\sharp}(x)$ divisible by $v_n^B$ for $n\geq n_y$.
Consequently $\gamma^{\sharp}(x) \leq \gamma^{\sharp}(y)$.
If $n_x = n_y$, then $x/v_{n_x}^A \leq y/v_{n_y}^A$, and we replace $x,y$ with $x/v_{n_x}^A \leq y/v_{n_y}^A$ in the above argument;
the obvious downward induction now implies that $\gamma^{\sharp}(x)\leq \gamma^{\sharp}(y)$.
\end{proof}

\begin{prop} \label{comparison map is injective}
Let $K,L$ be $p$-adic number fields with rings of
integers $A,B$, and let $L/K$ be a totally ramified, 
finite extension. 
Then $V^A\otimes_A B\stackrel{\gamma^{\sharp}}{\longrightarrow}V^B$ is injective.
\end{prop}\begin{proof} 
Let $\pi_A,\pi_B$ be uniformizers for $A,B$ respectively.
By equation~\ref{totramgammamodideal}, for each $n$, $\gamma^{\sharp}(v_n^A) \equiv \frac{\pi_A}{\pi_B} v_n^B$ modulo $(v_1^B, v_2^B, \dots , v_{n-1}^B)$.
Suppose that $f\in V^A$ is nonzero.
Order the monomials in $f$ using the lexicographic order from Definition~\ref{vaordering}.
Since this is a total ordering, there exists a monomial $g$ of $f$ which is maximal. 
By equation~\ref{hazewinkel log coeff eq 2}, $\gamma$ is nonvanishing on each monomial in $V^A$, so $\gamma^{\sharp}(g) \neq 0$.
Choose a maximal monomial $\overline{g}$ of $\gamma^{\sharp}(g)$. Adding lesser (in the partial ordering) monomials to $\overline{g}$ cannot yield zero,
so Lemma~\ref{gamma preserves ordering} implies that $\gamma^{\sharp}(f) \neq 0$. So $\gamma^{\sharp}$ is injective.
\end{proof}

\begin{theorem}\label{comparison map is a rational iso}
Let $K/\mathbb{Q}_p$ be a totally ramified finite extension,
and let $A$ be the ring of integers of $K$.
Then the resulting map $BP_*\otimes_{\mathbb{Z}_{(p)}} A  \stackrel{
}{\longrightarrow} V^A$
is a rational isomorphism, i.e., the graded ring homomorphism
\[ p^{-1}(BP_*\otimes_{\mathbb{Z}_{(p)}} A) \stackrel{
}{\longrightarrow} p^{-1}V^A \]
is an isomorphism.
\end{theorem}
\begin{proof}
The map $BP_*\otimes_{\mathbb{Z}_{(p)}} \hat{\mathbb{Z}}_p \rightarrow V^{\hat{\mathbb{Z}}_p}$ classifying the underlying $p$-typical formal group law of the universal $\hat{\mathbb{Z}}_p$-typical formal $\hat{\mathbb{Z}}_p$-module is an isomorphism, by Hazewinkel's computation of $V^R$ for an arbitrary discrete valuation ring $R$ with finite residue field; see section~21.5 of~\cite{MR2987372}.
The map $\gamma^{\sharp}: V^{\hat{\mathbb{Z}}_p} \otimes_{\hat{\mathbb{Z}}_p} A \rightarrow V^A$ is injective by Proposition~\ref{comparison map is injective},
and it is rationally surjective by Proposition~\ref{rational surjectivity}.
\end{proof}

\section{Nonrealizability in the not-totally-ramified case.}\label{non-tot-ram section}

When $A$ is a discrete valuation ring with finite residue field, I write $V^A$ for the classifying ring of $A$-typical formal $A$-modules, which as Hazewinkel proved, is isomorphic to a polynomial $A$-algebra on countably infinitely many generators: $V^A \cong A[v_1^A, v_2^A, \dots ]$ (see chapter 21 of~\cite{MR2987372} for this fact).

Here is an easy observation: the spectrum $H\mathbb{F}_p$ has the property that its $BP$-homology $BP_*(H\mathbb{F}_p)$,
as a graded $BP_*$-module, splits as a direct sum of suspensions of the residue field $\mathbb{F}_p$ of $BP_*$.
Consequently, if $A$ is the ring of integers in a degree $d$ unramified field extension of $\mathbb{Q}_p$, then 
$BP_*(H\mathbb{F}_p^{\vee d})$ admits the structure of a $V^A$-module, since the residue field of $V^A$ is $\mathbb{F}_{p^d}$,
which is isomorphic to $\mathbb{F}_p^{\oplus d}$ as a $BP_*$-module, and we can simply let all the generators
$\pi, v_1^A, v_2^A, \dots$ of $V^A$ act trivially on $BP_*(H\mathbb{F}_p^{\vee d})$.

This is a pretty silly example, though: $H\mathbb{F}_p$ is a {\em dissonant} spectrum (see Remark~\ref{dissonance remark} for a brief review
of dissonant spectra and their pathological properties), so one cannot use
(and does not need) chromatic techniques to learn anything at all about $H\mathbb{F}_p$. One would like to know if there exist
any non-dissonant spectra $X$ such that $BP_*(X)$ is the underlying $BP_*$-module of a $V^A$-module.

In this section we see that the answer is {\em no}: if $X$ is a spectrum and $BP_*(X)$ is the underlying $BP_*$-module of a $V^A$-module,
where $A$ is the ring of integers of some field extension $K/\mathbb{Q}_p$, then either $K/\mathbb{Q}_p$ is totally ramified or
$X$ is dissonant (this is Theorem~\ref{unramified dissonance thm}). This tells us that $BP_*(X)$ cannot be a $V^A$-module,
for $K/\mathbb{Q}_p$ not totally ramified, unless $X$ has some rather pathological properties; for example, $BP_*(X)$ cannot be finitely-presented
as a $BP_*$-module (this is Corollary~\ref{unramified case and finite presentation}), and by 
Corollary~\ref{suspension spectra and VA cor}, the $BP$-homology of a {\em space} which is not $p$-locally contractible is
never a $V^A$-module unless $K/\mathbb{Q}_p$ is totally ramified, due to the Hopkins-Ravenel result~\cite{MR1317578} that suspension spectra
are harmonic. Finally, Corollary~\ref{unramified VA nonrealizability} gives us that $V^A$ itself is not the $BP$-homology of any spectrum. 

Before we begin with new results, here are three known results we will use, Lemma~\ref{bousfield class equality} and Theorems~\ref{ravenels local conj} and~\ref{smashing conj}:
\begin{lemma}\label{bousfield class equality}
The Bousfield class of $L_{E(n)}BP$ is equal to the Bousfield class of $v_n^{-1}BP$, the telescope of $v_n$ on $BP$.
\end{lemma}
\begin{proof}
This is Lemma~8.1.4 in Ravenel's book~\cite{MR1192553}.
\end{proof}
\begin{theorem} \label{ravenels local conj}{\bf (Ravenel's localization conjecture; proven by Hopkins and Ravenel.)} 
Let $Y$ be a spectrum and $n$ a positive integer.
Then $BP\smash L_{E(n)} Y \cong Y \smash L_{E(n)}BP$.
Consequently, 
if $v_{n-1}^{-1}BP_*(Y) \cong 0$, then 
$BP_*(L_{E(n)}Y) \cong v_n^{-1}BP_*(Y)$.
\end{theorem}
\begin{proof} See section 8.1 of~\cite{MR1192553}.
\end{proof}

\begin{theorem} \label{smashing conj}{\bf (Ravenel's smashing conjecture; proven by Hopkins and Ravenel.)} 
Let $n$ be a nonnegative integer. Then every $p$-local spectrum is $E(n)$-prenilpotent. Consequently:
\begin{itemize}
\item $E(n)$-localization is smashing, i.e., $L_{E(n)} X \simeq L_{E(n)}S \smash X$ for all $p$-local spectra $X$, and
\item $E(n)$-nilpotent completion coincides with $E(n)$-localization, i.e., $\hat{X}_{E(n)} \simeq L_{E(n)}X$ for all $p$-local spectra $X$.
\end{itemize}
\end{theorem}
\begin{proof} See section 8.2 of~\cite{MR1192553}.
\end{proof}

\begin{definition}\label{def of eventual div}
Let $R$ be a  graded ring, and let $M$ be a graded left $R$-module. Let $r\in R$ be a homogeneous element.
\begin{itemize}
\item We say that $M$ is {\em $r$-power-torsion}
if, for each homogeneous $x\in M$, there exists some nonnegative integer $m$ such
that $r^mx = 0$.
\item 
We say that {\em $r$ acts trivially on $M$} if $rx = 0$ for all homogeneous $x\in M$.
\item 
We say that {\em $r$ acts without torsion on $M$} if for all homogeneous nonzero $x\in M$, $rx \neq 0$.
\item
We say that {\em $r$ acts invertibly on $M$} if, for each homogeneous $x\in M$, there exists
some $y\in M$ such that $ry = x$.
\item
Suppose $s\in R$ is a homogeneous element.
We say that {\em $r$ acts with eventual $s$-division on $M$} if, 
for each homogeneous $x\in M$, there exists some nonnegative integer $n$ and some $y\in M$ such that $sy = r^nx$.
\item
Suppose $s\in R$ is a homogeneous element, and $I\subseteq R$ is a homogeneous ideal.
We say that {\em $r$ acts with eventual $s$-division on $M$ modulo $I$} if, 
for each homogeneous $x\in M$, there exists some nonnegative integer $n$ and some $y\in M$ such that $sy \equiv r^nx$ modulo $IM$.
\end{itemize}
\end{definition}

Lemmas~\ref{acyclicity lemma} and~\ref{finite presentation and power-torsion} and Proposition~\ref{dissonance lemma} are all easy and surely well-known, but I do not know where to find any of them in the literature:
\begin{lemma}\label{acyclicity lemma}
Let $X$ be a spectrum, and let $n$ be a nonnegative integer.
Then the Bousfield localization $L_{E(n)}X$ is contractible
if and only if $BP_*(X)$ is $v_n$-power-torsion.
\end{lemma}
\begin{proof}
Suppose that $BP_*(X)$ is $v_n$-power-torsion.
Since $E(n)_*$ is a Landweber exact homology theory (see~\cite{MR0423332}),
we have isomorphisms of $BP_*$-modules
\begin{align*}  
 E(n)_*(X) 
  &\cong E(n)_* \otimes_{BP_*} BP_*(X) \\
  &\cong v_n^{-1}\left( BP_*(X) / (v_{n+1},v_{n+2}, \dots ) BP_*(X)\right) \end{align*}
and $BP_*(X)/(v_{n+1},v_{n+2}, \dots ) BP_*(X)$ is $v_n$-power-torsion by Lemma~\ref{quotient of power torsion is power torsion}. So $E(n)_*(X) \cong 0$, so $X$ is $E(n)$-acyclic, so $L_{E(n)}X$ is contractible.

%

On the other hand, if $L_{E(n)}X \cong 0$, 
then $BP \smash L_{E(n)} X$ is contractible, 
hence $L_{E(n)}BP \smash X$ is contractible by Theorem~\ref{ravenels local conj}.
Hence $(v_n^{-1}BP)\smash X$ is contractible by Lemma~\ref{bousfield class equality}.
Now, since $v_n^{-1}BP_*$ is a Landweber-exact homology theory,
we have
\begin{align*}
 0 
   \cong \pi_*\left( (v_n^{-1}BP)\smash X\right) \\
   \cong  \left( v_n^{-1} BP_*\right) \otimes_{BP_*} BP_*(X).
\end{align*}
So inverting $v_n$ kills $BP_*(X)$, so $BP_*(X)$ is $v_n$-power-torsion.
\end{proof}

\begin{remark}\label{dissonance remark}
Recall (from e.g.~\cite{MR737778}) that a spectrum $X$ is said to be {\em dissonant at $p$} (or just {\em dissonant}, if the prime $p$ is understood from the context) if it is acyclic with respect to the wedge 
$\bigvee_{n\in \mathbb{N}} E(n)$ of all the $p$-local Johnson-Wilson theories, equivalently, if it is 
acyclic with respect to the wedge 
$\bigvee_{n\in \mathbb{N}} K(n)$ of all the $p$-local Morava $K$-theories. Dissonant spectra are ``invisible'' to the typical methods of chromatic 
localization. In some sense dissonant spectra are uncommon; for example, it was proven in~\cite{MR1317578}
that all $p$-local suspension spectra are {\em harmonic,} that is, local with respect to the wedge of the $p$-local Johnson-Wilson theories,
hence as far from dissonant as possible. On the other hand, there are some familiar examples of dissonant spectra, for example,
the Eilenberg-Mac Lane spectrum $HA$ of any torsion abelian group $A$ (this example appears in~\cite{MR737778}).
\end{remark}
\begin{prop}\label{dissonance lemma}
Let $p$ be a prime number and let $X$ be a spectrum. Then the following conditions are equivalent:
\begin{itemize}
\item $X$ is dissonant at $p$.
\item $BP_*(X)$ is $v_n$-power-torsion for infinitely many $n$. 
\item $BP_*(X)$ is $v_n$-power-torsion for all $n$. 
\end{itemize}
\end{prop}
\begin{proof}
\begin{itemize}
\item Suppose that $BP_*(X)$ is $v_n$-power-torsion for infinitely many $n$.
For each nonnegative integer $m$, there exists some integer $n$ such that $n\geq m$ and such that $BP_*(X)$ is $v_n$-power-torsion.
Hence, by Lemma~\ref{acyclicity lemma}, $L_{E(n)}X$ is contractible. Hence 
\begin{align*} L_{E(m)}X \simeq  L_{E(m)}L_{E(n)} X  \simeq  L_{E(m)}0  \simeq  0 . \end{align*}
Hence $X$ is $E(m)$-acyclic for all $m$. Hence 
\begin{align*}
 \left( \bigvee_{n\in \mathbb{N}} E(n) \right) \smash X \simeq  \bigvee_{n\in \mathbb{N}} \left( E(n)  \smash X \right) \simeq  \bigvee_{n\in \mathbb{N}} 0 \simeq   0 .\end{align*}
So $X$ is dissonant.
\item 
Now suppose that $X$ is dissonant. Then it is in particular 
$E(0)$-acyclic, so $BP\smash L_{E(0)}X$ must be contractible, so $p^{-1}BP_*(X)$ must be trivial.
So $BP_*(X)$ must be $p$-power-torsion.
This is the initial step in an induction: suppose we have already shown that
$BP_*(X)$ is $v_i$-power-torsion for all $i< n$. Then 
Ravenel's localization conjecture/theorem (see Theorem~7.5.2 of \cite{MR1192553}, or Theorem~\ref{ravenels local conj} in the present paper for the statement of the result)
gives us that $BP_*(L_{E(n)}X) \cong v_n^{-1}BP_*(X)$, and since $X$ is dissonant, 
$L_{E(n)}X$ must be contractible, so $v_n^{-1}BP_*(X)$ must be trivial. Hence $BP_*(X)$ must be $v_n$-power-torsion.
Now by induction, $BP_*(X)$ is $v_n$-power-torsion for all $n$.
\item Clearly if $BP_*(X)$ is $v_n$-power-torsion for all $n$, then it is $v_n$-power-torsion for infinitely many $n$.
\end{itemize}
\end{proof}

\begin{lemma}\label{finite presentation and power-torsion}
Let $A$ be a local commutative ring, let $g: \mathbb{N} \rightarrow \mathbb{N}$ be a monotone strictly increasing function, 
and let $R$ be the commutative graded ring
\[ R = A[x_0, x_1, x_2, \dots ] \]
with $x_i$ in grading degree $g(i)$.
Suppose $M$ is a nonzero finitely-presented graded $R$-module. Then there exists some $n$ such that $M$ is not $x_n$-power-torsion.
\end{lemma}
\begin{proof}
Choose a presentation
\begin{equation}\label{presentation 20} \coprod_{s\in S_1} \Sigma^{d_1(s)} R\{ e_s\} \stackrel{f}{\longrightarrow} \coprod_{s\in S_0} \Sigma^{d_0(s)} R\{ e_s\} \rightarrow M \rightarrow 0\end{equation}
for $M$, where $S_1,S_0$ are finite sets and $d_0: S_0 \rightarrow \mathbb{Z}$ and $d_1: S_1\rightarrow \mathbb{Z}$ are simply to keep track of the
grading degrees of the summands $R\{ e_s\}$.
Now for each $s_1\in S_1$ and each $s_0\in S_0$, the component of $f(e_{s_1})$ in the summand $R\{e_{s_1}\}$ of the codomain of $f$ is
some grading-homogeneous polynomial in $R$ of grading degree $d_1(s_1) - d_0(s_0)$. Since $g$ was assumed monotone strictly increasing,
there are only finitely many $x_i$ such that $g(x_i) \leq d_1(s_1) - d_0(s_0)$, so there exists some maximal integer $i$
such that $x_i$ appears in one of the monomial terms of the component of $f(e_{s_1})$ in the summand $R\{e_{s_1}\}$ of the codomain of $f$.
Let $i_{s_0,s_1}$ denote that integer $i$.

Now there are only finitely many elements $s_0$ of $S_0$ and only finitely many elements $s_1$ of $S_1$, so
the maximum $m = \max\{ i_{s_0, s_1}: s_0 \in S_0, s_1\in S_1\}$ is some integer.
Let $\tilde{R}$ denote the commutative graded ring
\[ \tilde{R} = A[x_0, x_1, x_2, \dots , x_m ],\]
with $x_i$ again in grading degree $g(x_i)$,
and regard $R$ as a graded $\tilde{R}$-algebra by the evident ring map $\tilde{R}\rightarrow R$.
then the map $f$ from presentation~\ref{presentation 20} is isomorphic to
\begin{equation}\label{presentation 21} \coprod_{s\in S_1} \Sigma^{d_1(s)} \tilde{R}\{ e_s\} \stackrel{\tilde{f}}{\longrightarrow} \coprod_{s\in S_0} \Sigma^{d_0(s)} \tilde{R}\{ e_s\} \end{equation}
tensored over $\tilde{R}$ with $R$.
Here $\tilde{f}(e_s)$ is defined to be have the same components as $f(e_s)$.
Since $R$ is flat over $\tilde{R}$, we now have that $M\cong R\otimes_{\tilde{R}} \tilde{M}$, where $\tilde{M}$ is the cokernel of~\ref{presentation 21}. Hence 
$x_i$ acts without torsion on $M$ for all $i>m$.
\end{proof}

\begin{theorem}\label{unramified dissonance thm}
Let $K/\mathbb{Q}_p$ be a finite extension of degree $>1$, and which is not totally ramified.
Let $A$ denote the ring of integers of $K$, and 
let $X$ be a spectrum such that $BP_*(X)$ is the underlying $BP_*$-module of 
a $V^A$-module.
Then $X$ is dissonant, and $BP_*(X)$ is $v_n$-power-torsion for all $n\geq 0$.
\end{theorem}
\begin{proof}
Let $f$ denote the residue degree of the extension $K/\mathbb{Q}_p$.
Then the natural ring homomorphism $BP_*\rightarrow V^A$ 
factors as the composite of the natural maps
\[ BP_* \rightarrow V^{A_{nr}} \rightarrow V^A,\]
where $A_{nr}$ is the ring of integers of the maximal unramified subextension $K_{nr}/\mathbb{Q}$ of $K/\mathbb{Q}$.
By Proposition~\ref{computationofunramifiedgamma},
the natural map $BP_* \rightarrow V^{A_{nr}}$
sends $v_n$ to zero if $f$ does not divide $n$.
Hence the natural ring homomorphism $BP_*\rightarrow V^A$ sends $v_n$ to zero if $f$ does not divide $n$.
Hence, if $BP_*(X)$ is the underlying $BP_*$-module of some $V^A$-module, then $BP_*(X)$ is $v_n$-power-torsion for all $n$ not divisible by $f$.
Lemma~\ref{dissonance lemma} then implies that $X$ is dissonant. Now Proposition~\ref{dissonance lemma} implies that
$BP_*(X)$ is $v_n$-power-torsion for all $n\geq 0$.
\end{proof}

\begin{corollary}\label{unramified VA nonrealizability}
Let $K/\mathbb{Q}_p$ be a finite extension of degree $>1$, and which is not totally ramified.
Let $A$ denote the ring of integers of $K$.
Then there does not exist a spectrum $X$ such that $BP_*(X) \cong V^A$.
\end{corollary}
\begin{proof}
It is not the case that $V^A$ is $p$-power-torsion.
\end{proof}

\begin{corollary}\label{unramified case and finite presentation}
Let $K/\mathbb{Q}_p$ be a finite extension extension of degree $>1$ which is not totally ramified.
Let $A$ denote the ring of integers of $K$.
Suppose $M$ is a graded $V^A$-module, and suppose that
$X$ is a spectrum such that $BP_*(X) \cong M$. 
Suppose that either
\begin{itemize}
\item $M$ is finitely presented as a $BP_*$-module, or
\item $M$ is finitely presented as a $V^A$-module, and $K/\mathbb{Q}_p$ is unramified.
\end{itemize}
Then $M\cong 0$.
\end{corollary}
\begin{proof}
By Theorem~\ref{unramified dissonance thm}, $M$ must be $v_n$-power-torsion for all $n\geq 0$.
Lemma~\ref{finite presentation and power-torsion} then implies that $M$ cannot be finitely presented as a $BP_*$-module unless $M \cong 0$.

If $K/\mathbb{Q}_p$ is unramified, then $V^A \cong A[v_{d}, v_{2d}, v_{3d}, \dots ]$ as
a $BP_*$-module, by Proposition~\ref{computationofunramifiedgamma}. Hence, if $M$ is a graded $V^A$-module which is $v_{n}$-power-torsion for all integers $n\geq 0$,
then by Lemma~\ref{finite presentation and power-torsion}, 
$M$ cannot be finitely presented  as a $V^A$-module unless $M\cong 0$.
\end{proof}

\begin{corollary}\label{suspension spectra and VA cor}
Let $K/\mathbb{Q}_p$ be a finite extension of degree $>1$, and which is not totally ramified.
Let $A$ denote the ring of integers of $K$.
Then the $BP$-homology of a {\em space} is never the underlying $BP_*$-module of a $V^A$-module, unless that
space is stably contractible.
\end{corollary}
\begin{proof}
By~\cite{MR1317578}, the suspension spectrum of any space is harmonic. By Theorem~\ref{unramified dissonance thm},
if $BP_*(\Sigma^{\infty} X)$ is a $V^A$-module, then $\Sigma^{\infty} X$ is dissonant. The only spectra that are both dissonant and harmonic
are the contractible spectra.
\end{proof}

\begin{corollary}\label{suspension spectra and VA cor 2}
Let $K/\mathbb{Q}_p$ be a finite extension of degree $>1$, and which is not totally ramified.
Let $A$ denote the ring of integers of $K$.
If $Y$ is a $p$-local spectrum whose $BP$-homology is the underlying $BP_*$-module of a $V^A$-module, 
then every map from $Y$ to a suspension spectrum is nulhomotopic. In particular, the stable cohomotopy groups $\pi^*(Y)$ of $Y$ are all zero, and the Spanier-Whitehead dual $DY$ of $Y$ is contractible.
\end{corollary}
\begin{proof}
By~\cite{MR1317578}, the suspension spectrum of any space is harmonic. By Theorem~\ref{unramified dissonance thm},
if $BP_*(Y)$ is a $V^A$-module, then $Y$ is dissonant. 
Hence $Y$ is acyclic with respect to the wedge of the $p$-primary Johnson-Wilson theories, while any suspension spectrum $\Sigma^{\infty} X$ is local with respect to the same
wedge. Hence every map from $Y$ to $\Sigma^{\infty} X$ is nulhomotopic.
\end{proof}

\section{Nonrealizability in the totally ramified case.}\label{tot ram case}

Here is another observation, parallel to the one at the beginning of the previous section: 
the spectrum $H\mathbb{F}_p$ has the property that its $BP_*$-homology $BP_*(H\mathbb{F}_p)$,
as a graded $BP_*$-module, splits as a direct sum of suspensions of the residue field $\mathbb{F}_p$ of $BP_*.$
Consequently, if $A$ is the ring of integers in a totally ramified field extension of $\mathbb{Q}_p$, then 
$BP_*(H\mathbb{F}_p)$ admits the structure of a $V^A$-module, since the residue field of $V^A$ is $\mathbb{F}_{p}$, 
and we can simply let all the generators
$\pi, v_1^A, v_2^A, \dots$ of $V^A$ act trivially on $BP_*(H\mathbb{F}_p)$.

However, unlike the unramified situation, in the totally ramified setting 
there is another obvious source of spectra $X$ such that $BP_*(X)$ is a $V^A$-module:
if $d$ is the degree of the field extension $K/\mathbb{Q}_p$ of which $A$ is the ring of integers, then
the Eilenberg-Mac Lane spectrum $H\mathbb{Q}_p^{\vee d}$ admits the structure of a $V^A$-module, since
\[ BP_*(H\mathbb{Q}_p^{\vee d}) \cong (\mathbb{Q}_p\otimes_{\mathbb{Z}} BP_*)^{\oplus d} \cong p^{-1}V^A\]
as $BP_*$-modules, since the map $BP_*\otimes_{\mathbb{Z}_{(p)}} A \stackrel{\cong}{\longrightarrow} V^{\mathbb{Z}_{(p)}}\otimes_{\mathbb{Z}_{(p)}} A \rightarrow V^A$ is a rational isomorphism whenever $K/\mathbb{Q}_p$ is totally ramified, by Theorem~\ref{comparison map is a rational iso}.

So, by taking dissonant spectra and rational spectra, we can produce easy examples of spectra $X$ such that $BP_*(X)$ is the underlying
$BP_*(X)$-module of a $V^A$-module. These examples are not very interesting: again, one does not need chromatic localization methods 
to study rational spectra, and one cannot use chromatic localization methods to study dissonant spectra.
So one would like to know if there exist
any spectra $X$ which are not ``built from'' rational and dissonant spectra, and 
such that $BP_*(X)$ is the underlying $BP_*$-module of a $V^A$-module.

In this section we see that the answer is {\em no}: if $X$ is a spectrum and $BP_*(X)$ is the underlying $BP_*$-module of a $V^A$-module,
where $A$ is the ring of integers of some totally ramified field extension $K/\mathbb{Q}_p$, then 
$X$ is an extension of a rational spectrum by a dissonant spectrum (this is Theorem~\ref{main thm on tot ram nonexistence}).
This tells us that $BP_*(X)$ cannot be a $V^A$-module,
for $K/\mathbb{Q}_p$ totally ramified, unless $X$ has some rather pathological properties; 
for example, if $BP_*(X)$ is torsion-free as an abelian group, then it cannot be finitely-presented
as a $BP_*$-module or as a $V^A$-module (this is Theorem~\ref{torsion-free and local coh not fin pres}), and by 
Corollary~\ref{tot ram corollary for spaces}, the $BP$-homology of a {\em space} is 
not a $V^A$-module unless that space is stably rational, i.e., its stable
homotopy groups are $\mathbb{Q}$-vector spaces. Finally, Corollary~\ref{nonrealizability of VA} gives us that $V^A$ itself is not the $BP$-homology of any spectrum.

First, we need a little bit of easy algebra, Lemmas~\ref{quotient of power torsion is power torsion}, \ref{coproduct of power torsion is power torsion}, \ref{localization preserves power torsion property}, \ref{power torsion preserved under extension}, and \ref{modules over eventually div algs are eventually div}:
\begin{lemma}\label{quotient of power torsion is power torsion}
Let $R$ be a commutative graded ring, $r\in R$ a homogeneous element,
and let $M$ be a graded $R$-module which is $r$-power-torsion.
\begin{itemize}
\item If $M \stackrel{f}{\longrightarrow} N$ be a surjective homomorphism of graded $R$-modules, then $N$ is $r$-power-torsion.
\item If $N \stackrel{f}{\longrightarrow} M$ is an injective homomorphism of graded $R$-modules, then $N$ is $r$-power-torsion.
\end{itemize}
\end{lemma}
\begin{proof}
Very easy.
\end{proof}

\begin{lemma}\label{coproduct of power torsion is power torsion}
Let $R$ be a commutative graded ring, and let $r\in R$ be a homogeneous element.
Then any direct sum of $r$-power-torsion graded $R$-modules is also $r$-power-torsion.
\end{lemma}
\begin{proof}
Let $S$ be a set and, for each $s\in S$, let $M_s$ be an $r$-power-torsion graded $R$-module.
Any homogeneous element $m$ of the direct sum $\coprod_{s\in S} M_s$ has only finitely many nonzero components;
for each such component $m_s$, choose a nonnegative integer $i_s$ such that $r^{i_s}m_s = 0$.
Then the maximum $i =\max\{ i_s: s\in S\}$ satisfies $r^i m_s = 0$ for all $s\in S$, hence $r^i m = 0$.
Hence $\coprod_{s\in S} M_s$ is $r$-power-torsion.
\end{proof}

\begin{lemma}\label{localization preserves power torsion property}
Let $R$ be a commutative graded ring, $r,s\in R$ homogeneous elements,
and let $M$ be a graded $R$-module which is $r$-power-torsion.
Then $s^{-1}M$ is $r$-power-torsion.
\end{lemma}
\begin{proof}
I will write $\coprod_{n\in\mathbb{N}} M\{ e_n\}$ to mean the countably infinite coproduct $\coprod_{n\in\mathbb{N}} M$ but equipped with a choice of
labels $e_0, e_1, e_2, \dots$ for the summands of $\coprod_{n\in\mathbb{N}} M$.

The module $s^{-1}M$ is isomorphic to the cokernel of the map
\[ \coprod_{n\in\mathbb{N}} M\{ e_n\} \stackrel{\id - T}{\longrightarrow} \coprod_{n\in\mathbb{N}} M\{ e_n\}\]
where $T$ is the map of graded $R$-modules
given, for each $m\in \mathbb{N}$, on the component $M\{ e_m\}$ by the composite map
\[ M\{ e_m\} \stackrel{s}{\longrightarrow} M\{ e_m\} \stackrel{\cong}{\longrightarrow} M\{e_{m+1}\} \stackrel{i}{\longrightarrow} \coprod_{n\in\mathbb{N}} M\{ e_n\},\]
where $i$ is the direct summand inclusion map. 

Now $\coprod_{n\in\mathbb{N}} M\{ e_n\}$ is $r$-power-torsion by Lemma~\ref{coproduct of power torsion is power torsion},
so $s^{-1}M$ is a quotient of an $r$-power-torsion module.
So $s^{-1}M$ is $r$-power-torsion by Lemma~\ref{quotient of power torsion is power torsion}.
\end{proof}

\begin{lemma}\label{power torsion preserved under extension}
Let $R$ be a graded ring, and let $r\in R$ be a homogeneous element.
Let \[ 0 \rightarrow M^{\prime}\stackrel{f}{\longrightarrow} M \stackrel{g}{\longrightarrow} M^{\prime\prime}  \rightarrow 0\]
be a short exact sequence of graded $R$-modules, and suppose that
$M^{\prime}$ and $M^{\prime\prime}$ are both $r$-power-torsion. Then $M$ is also $r$-power-torsion.
\end{lemma}
\begin{proof}
Let $m\in M$. Then $g(r^n m) = r^ng(m) = 0$ for some nonnegative integer $n$ since $M^{\prime\prime}$ is $r$-power-torsion, so there exists some $\tilde{m}\in M^{\prime}$ such that $f(\tilde{m}) = r^nm$. Then $r^{\ell}\tilde{m} = 0$ for some nonnegative integer $\ell$, since $M^{\prime}$
is $r$-power-torsion. So $0 = f(r^{\ell}\tilde{m}) = r^{\ell}r^nm = r^{\ell+n}m$. So $M$ is $r$-power-torsion.
\end{proof}

\begin{lemma}\label{modules over eventually div algs are eventually div}
Let $R,S$ be a commutative graded ring, let $I$ be a homogeneous ideal of $R$, let $S$ be a commutative graded $R$-algebra, and let $r,r^{\prime}$ be homogeneous elements of $R$ such that $r$ acts with eventual $r^{\prime}$-division on $S$ modulo $I$. Then, for every graded $S$-module $M$, $r$ acts with eventual $r^{\prime}$-division on $S$ modulo $I$.
\end{lemma}
\begin{proof}
Choose some nonnegative integer $n$ and homogeneous element $y\in S$ such that $r^{\prime} y \equiv r^n 1$ modulo $IS$.
Then, for each homogeneous element $x$ of $M$, the element $yx\in M$ satisfies
$r^{\prime}(yx) \equiv r^n x$ modulo $IM$.
\end{proof}

To my mind, Lemma~\ref{localization and eventual division} is a curious fact: 
it implies that, for example, the graded $BP_*$-module $BP_*/(p, v_2-v_1^{p+1})$ is not the $BP$-homology of any spectrum, and more generally,
$BP_*/(p^a, v_2^b - v_1^{(p+1)b})$ is not the $BP$-homology of any spectrum for any positive integers $a,b$. While there is a tremendous wealth of graded $BP_*$-modules, Lemma~\ref{localization and eventual division} implies that vastly many of them are not realizable as the $BP$-homology of any spectrum.
\begin{lemma}\label{localization and eventual division}
Suppose that $p$ is a prime and $n$ is a positive integer, and suppose that 
$X$ is a $p$-local spectrum such that $BP_*(X)$ is $v_i$-power-torsion for $i=0, 1, \dots ,n-1$.
Suppose that 
$v_{n+1}$ acts with eventual $v_n$-division on $BP_*(X)$ modulo $I_n$.
Then $BP_*(X)$ is $v_n$-power-torsion, and 
the Bousfield localization $L_{E(n)}X$ is contractible.
\end{lemma}
\begin{proof}
First, since $I_n$ acts trivially on $v_n^{-1}BP_*(X)/I_nBP_*(X)$, 
the Miller-Ravenel change-of-rings isomorphism (see section 6.1 of~\cite{MR860042}, or~\cite{MR0458410} for the original reference),
gives us an isomorphism of bigraded $\mathbb{Z}_{(p)}$-modules
\begin{align*}
 &  \Ext_{(BP_*,BP_*BP)}^{*,*}(BP_*,v_n^{-1}BP_*(X)/I_nBP_*(X)) \\
  & \cong  \Ext_{(E(n)_*,E(n)_*E(n))}^{*,*}(E(n)_*,(BP_*(X)/I_nBP_*(X))\otimes_{BP_*} E(n)_*). \end{align*}
Now 
\[ E(n)_* \cong v_n^{-1}BP_*/(v_{n+1}, v_{n+2}, \dots ),\]
hence 
\begin{align*} (BP_*(X)/I_nBP_*(X))\otimes_{BP_*}E(n)_* &\cong v_n^{-1}BP_*(X)/(p, v_1, \dots ,v_{n-1}, v_{n+1}, v_{n+2}, \dots )BP_*(X).
\end{align*}

Now if $v_n$ is assumed to act with eventual $v_{n+1}$-division on $BP_*$ modulo $I_n$, then
for any given element of 
$BP_*(X)/(p, v_1, \dots ,v_{n-1}, v_{n+1}, v_{n+2}, \dots )BP_*(X)$ 
there is some power of $v_n$ that acts trivially on that element,
and hence 
\begin{align} \nonumber v_n^{-1}BP_*(X)/(p, v_1, \dots ,v_{n-1}, v_{n+1}, v_{n+2}, \dots )BP_*(X) &\cong 0,\mbox{\ \ so} \\
\label{iso 100000}\Ext_{(BP_*,BP_*BP)}^{*,*}(BP_*,v_n^{-1}BP_*(X)/I_nBP_*(X))&\cong 0.\end{align}

I claim that this implies that $\Ext_{(BP_*,BP_*BP)}^{*,*}(BP_*,v_n^{-1}BP_*(X))$ also 
vanishes; here is a proof.
For each integer $j$ such that $0 < j \leq n$,
filter $v_n^{-1}BP_*(X)/I_j$ by kernels of multiplication by powers of $v_j$:
let $F^iBP_*(X)$ be the kernel of the map
\begin{equation}\label{comodule map 1} v_j^i : v_n^{-1}BP_*(X)/I_j\rightarrow v_n^{-1}BP_*(X)/I_j,\end{equation}
so that we have the increasing filtration
\begin{equation}\label{comodule filtration 1} 0 = F^0v_n^{-1}BP_*(X)/I_j \subseteq F^1v_n^{-1}BP_*(X)/I_j \subseteq F^2v_n^{-1}BP_*(X)/I_j \subseteq \dots .\end{equation}
The filtration~\ref{comodule filtration 1} is separated/Hausdorff, since $\cap_i F^iBP_*(X) = 0$, and
complete as well, since $\lim_{i\rightarrow-\infty} (BP_*(X)/I_nBP_*(X))/F^i(BP_*(X)/I_nBP_*(X)) \cong BP_*(X)/I_nBP_*(X)$.
Furthermore, the map~\ref{comodule map 1} is a $BP_*BP$-comodule homomorphism, so filtration~\ref{comodule filtration 1} is a filtration of
$v_n^{-1}BP_*(X)/I_j$ by sub-$BP_*BP$-comodules.
We assumed that $BP_*(X)$ is $v_i$-power-torsion for $i=0, 1, \dots , n-1$,
so Lemma~\ref{quotient of power torsion is power torsion} and Lemma~\ref{localization preserves power torsion property} together imply that 
$v_n^{-1}BP_*(X)/I_j$ is $v_i$-power-torsion for $i=0, 1, \dots, n-1$ and for all $j$. Hence filtration~\ref{comodule filtration 1} is exhaustive. 

Now the usual functoriality properties of the cobar complex $C^{\bullet}_{(A,\Gamma)}(M)$ of a $\Gamma$-comodule $M$ then tell us that
\[ C^{\bullet}_{(BP_*,BP_*BP)}(F^0BP_*(X)) \subseteq 
 C^{\bullet}_{(BP_*,BP_*BP)}(F^1BP_*(X)) \subseteq 
 C^{\bullet}_{(BP_*,BP_*BP)}(F^2BP_*(X)) \subseteq \dots \]
is a complete, exhaustive, and separated/Hausdorff filtration on 
the cobar complex \linebreak $C^{\bullet}_{(BP_*,BP_*BP)}(v_n^{-1}BP_*(X)/I_j)$
computing $\Ext^{*,*}_{(BP_*,BP_*BP)}(BP_*, v_n^{-1}BP_*(X)/I_j)$.

Hence, from e.g. Theorem 9.2 of~\cite{MR1718076}, , we get a
conditionally convergent spectral sequence
\begin{align}\label{filtration ss 1} E_1^{s,t,u} & \cong \Ext^{s,u}_{(BP_*,BP_*BP)}(BP_*, (F^tv_n^{-1}BP_*(X)/I_j)/(F^{t-1}v_n^{-1}BP_*(X)/I_j)) \\
 &\Rightarrow \Ext^{s,u}_{(BP_*,BP_*BP)}(BP_*,v_n^{-1}BP_*(X)/I_j).\end{align}
We have an isomorphism of graded $BP_*BP$-comodules
\begin{equation}\label{iso 100001} (F^tv_n^{-1}BP_*(X)/I_j)/(F^{t-1}v_n^{-1}BP_*(X)/I_j) \cong v_n^{-1}BP_*(X)/I_{j+1}\end{equation}
for all $t$,
and consequently the $E_1$-term of the spectral sequence~\ref{filtration ss 1} is isomorphic (up to appropriate regrading) to a direct sum of 
countably infinitely many copies of \linebreak
$\Ext^{*,*}_{(BP_*,BP_*BP)}(BP_*, v_n^{-1}BP_*(X)/I_{j+1})$.

Now we just need an easy downward induction: 
isomorphisms~\ref{iso 100000} and~\ref{iso 100001} tell us that the $E_1$-term of spectral sequence~\ref{filtration ss 1} vanishes when $j=n-1$.
That was the initial step. For the inductive step,
suppose we already know that 
\[ \Ext^{*,*}_{(BP_*,BP_*BP)}(BP_*,v_n^{-1}BP_*(X)/I_j) \cong 0\]
for some $j\geq 0$; then spectral sequence~\ref{filtration ss 1} and
isomorphism~\ref{iso 100001} tell us that
\[\Ext^{*,*}_{(BP_*,BP_*BP)}(BP_*,v_n^{-1}BP_*(X)/I_{j-1}) \cong 0\] as well.
By downward induction, 
\[ \Ext^{*,*}_{(BP_*,BP_*BP)}(BP_*,v_n^{-1}BP_*(X)/I_0) \cong 0,\]
i.e. (since $I_0 = (0)$), 
\begin{equation}\label{iso 100002}\Ext^{*,*}_{(BP_*,BP_*BP)}(BP_*,v_n^{-1}BP_*(X)) \cong 0.\end{equation}

Now since $BP_*(X)$ is assumed to be $v_{n-1}$-power-torsion, we have that $v_{n-1}^{-1}BP_*(X)\cong 0$,
and hence Theorem~\ref{ravenels local conj}
gives us that $BP_*(L_{E(n)}X) \cong v_n^{-1}BP_*(X)$. 
Using triviality of the $\Ext$-group~\ref{iso 100002} together with another Morava-Miller-Ravenel change-of-rings
(since $BP_*(X)$ is $I_n$-nil, i.e., $v_i$-power-torsion for all $i < n$), we have
\begin{align} 0 
\nonumber & \cong  \Ext^{*,*}_{(BP_*,BP_*BP)}(BP_*, v_n^{-1}BP_*(X)) \\
\nonumber & \cong  \Ext^{*,*}_{(BP_*,BP_*BP)}(BP_*, BP_*(L_{E(n)}X)) \\
\nonumber & \cong  \Ext^{*,*}_{(E(n)_*,E(n)_*E(n))}(E(n)_*, BP_*(L_{E(n)}X)\otimes_{BP_*} E(n)_*) \\
\label{iso 100003} & \cong  \Ext^{*,*}_{(E(n)_*,E(n)_*E(n))}(E(n)_*, E(n)_*(L_{E(n)}X)) \\
\label{iso 100004} & \cong  \Ext^{*,*}_{(E(n)_*,E(n)_*E(n))}(E(n)_*, E(n)_*(X)) ,\end{align}
with isomorphism~\ref{iso 100003} due to $E(n)_*$ being a Landweber-exact $BP_*$-module.
The bigraded $\mathbb{Z}_p$-module~\ref{iso 100004} is the $E_2$-term of the $E(n)$-Adams spectral sequence
converging to $\hat{X}_{E(n)}$, the $E(n)$-nilpotent completion of $X$.
Hence $\hat{X}_{E(n)}$ is contractible. 
Consequently, by Theorem~\ref{smashing conj},
$0 \simeq \hat{X}_{E(n)} \simeq L_{E(n)}X$, as claimed.
Since the $E(n)$-localization of $X$ is contractible, Lemma~\ref{acyclicity lemma} gives us that $BP_*(X)$ is $v_n$-power-torsion.
\end{proof}

\begin{lemma}\label{eventual divisibility of VA-modules}
Let $K/\mathbb{Q}_p$ be a finite, totally ramified extension of degree greater than one.
Then, for each integer $n>1$, 
the element $v_{n+1}\in BP_*$ acts with eventual division by $v_n$ on $V^A$ modulo $I_n$. 
Furthermore, if the degree $[K : \mathbb{Q}_p]$ of the field extension is greater than two, then $v_{2}\in BP_*$ acts with eventual division by $v_1$ on $V^A$ modulo $I_1$.
\end{lemma}
\begin{proof}
Let $e= [K:\mathbb{Q}_p]$.
Let $\gamma: BP_*\rightarrow V^A$ denote the usual ring map (i.e., the one classifying the underlying $p$-typical formal group law of the universal $A$-typical formal $A$-module) by which we regard $V^A$ as a $BP_*$-module. Let $\ell_1, \ell_2, \dots$ denote the logarithm coefficients for the universal $p$-typical formal group law, 
and let $\ell_1^A, \ell_2^A, \dots$ denote the logarithm coefficients for the universal $A$-typical formal $A$-module.
Then $\gamma(\ell_i) = \ell_i^A$ for all $i$, since $K/\mathbb{Q}_p$ is totally ramified; see Remark~\ref{remark on log coefficients}.
Then, by the definition of the Hazewinkel generators of $BP_*$ and of $V^A$, we have the equations
\begin{align*} 
 p\ell_{n+1}
  & =  \sum_{j=0}^{n} (v_{n+1-j}^{p^j}) \ell_j \mbox{\ \ and} \\
 \pi \ell_{n+1}^A 
  & =  \sum_{j=0}^{n} (v_{n+1-j}^{p^j}) \ell_j^A .
\end{align*}
Applying $\gamma$, dividing by $p$ and by $\pi$, setting $\gamma(\ell_{n+1})$ equal to $\ell_{n+1}^A$, and solving for $\gamma(v_{n+1})$ yields the equation
\begin{align} 
\label{equality 400} \gamma(v_{n+1}) 
  & =  \frac{p}{\pi} v_{n+1}^A + \sum_{j=1}^{n}\left( \left(\frac{p}{\pi} (v_{n+1-j}^A)^{p^j} - \gamma(v_{n+1-j})^{p^j}\right) \ell_j^A\right). \end{align}

We now break into two cases:
\begin{itemize}
\item If $n=1$, then, by Proposition~\ref{low degree gamma computation}, we have:
\begin{align} 
 \gamma(v_{2}) 
\label{equality 401}  & =  \frac{p}{\pi} v_{2}^A + \frac{p}{\pi^2} (v_1^A)^{p+1}  - \frac{p^p}{\pi^{p+1}}\gamma(v_1)^{p+1} . \end{align}
Let $m$ be a power of $p$ such that $m>e$.
Then 
\begin{align*}
 \gamma(v_2^m) &\equiv \left( \frac{p}{\pi} v_2^A\right)^m + \left( \frac{p}{\pi^2} (v_1^A)^{p+1}\right)^m + \left( \frac{-p^p}{\pi^{p+1}}\gamma(v_1)^{p+1}\right)^m \\
 &\equiv \frac{p^m}{\pi^{2m}} (v_1^A)^{m(p+1)},\end{align*}
which is zero modulo $\pi$ 
if the ramification degree of the extension $K/\mathbb{Q}_p$ is greater than two. Hence, if $[K: \mathbb{Q}_p] > 2$, then $v_2$ acts with eventual 
division by $v_1$ on $V^A$ modulo $I_1 = (p)$,
since being zero is a trivial case of being divisible by $v_2$.
\item 
If $n>1$, then, for any $i<n$, both $\gamma(v_i)$ and any integral expression in $V^A$ divisible by $\ell_i^A$ in $\mathbb{Q}\otimes_{\mathbb{Z}} V^A$
are zero modulo $\gamma(I_n)V^A$.
Consequently~\ref{equality 400} reduces to 
\begin{align*} 
 \gamma(v_{n+1}) 
  &\equiv \frac{p}{\pi} v_{n+1}^A + \frac{p}{\pi} (v_1^A)^{p^n}\ell_n^A  
 \mod \gamma(I_n)V^A.
\end{align*}
Now let $m$ be any power of $p$ satisfying $m > e$. Since $\frac{p}{\pi}$ has positive $\pi$-adic valuation, $\left(\frac{p}{\pi}\right)^m \equiv 0$ modulo $p$, and consequently 
$\gamma(v_{n+1})^{m} \equiv \left(\frac{p}{\pi} v_{n+1}^A\right)^m + \left(\frac{p}{\pi} (v_1^A)^{p^n}\ell_n^A\right)^m \equiv 0$ modulo $\gamma(I_n)V^A$.
Hence $v_{n+1}$ acts with eventual division by $v_n$ on $V^A$ modulo $I_n$ (since, again, 
being zero is a trivial case of being divisible by $v_n$).
\end{itemize}
\end{proof}

Beginning in the proof of Theorem~\ref{main thm on tot ram nonexistence}, we refer (when convenient) to \cref{local cohomology appendix}, the appendix on basic ideas from local cohomology at the end of this paper.
\begin{theorem}\label{main thm on tot ram nonexistence}
Let $K/\mathbb{Q}_p$ be a finite, totally ramified extension of degree greater than one.
Let $X$ be a spectrum such that $BP_*(X)$ is the underlying graded $BP_*$-module of a graded $V^A$-module. 
Then the following statements are all true:
\begin{itemize}
\item For all nonnegative integers $n$, $E(n)_*(X)$ is a $\mathbb{Q}$-vector space.
\item For all nonnegative integers $n$, $L_{E(n)}X$ is a rational spectrum, i.e., a wedge of suspensions of $H\mathbb{Q}$.
\item For all positive integers $n$, $K(n)_*(X) \cong 0$.
\item $X$ is an extension of a rational spectrum by a dissonant spectrum. That is, $X$ sits in a homotopy fiber sequence
\begin{equation}\label{fiber seq 60} cX \rightarrow X \rightarrow LX ,\end{equation}
where $LX$ is a rational spectrum, that is, a wedge product of suspensions of copies of the Eilenberg-Mac Lane spectrum $H\mathbb{Q}$,
and $cX$ is a dissonant spectrum, that is, $E(n)\smash cX$ for all $n>0$.
(Here $LX$ is the usual rationalization of $X$, i.e., $LX \simeq X \smash H\mathbb{Q}$.)
\end{itemize}
\end{theorem}
\begin{proof}
Let $cX$ be as in the statement of the theorem, i.e., $cX$ is the homotopy fiber of the rationalization map $X \rightarrow LX$.
Since $BP_*(LX) \cong p^{-1}BP_*(X)$ and the
map induced in $BP$-homology by the rationalization agrees
up to isomorphism with the 
natural localization map
$i: BP_*(X) \rightarrow p^{-1}BP_*(X)$, we have a short exact sequence 
of graded $BP_*$-modules
\begin{equation}\label{ses 60} 0 \rightarrow \coker \Sigma i \rightarrow BP_*(cX) \rightarrow
 \ker i \rightarrow 0,\end{equation}
i.e., using Propositions~\ref{basic properties of local cohomology 1} and~\ref{basic properties of local cohomology 2}, 
a short exact sequence of graded $BP_*$-modules
\begin{equation}\label{ses 61} 0 \rightarrow H^1_{(p)}(\Sigma BP_*(X)) \rightarrow BP_*(cX) \rightarrow
 H^0_{(p)}(BP_*(X)) \rightarrow 0.\end{equation}
Even better, the short exact sequences~\ref{ses 60} and~\ref{ses 61} are short exact sequences of graded $V^A$-modules, since $p$ is not just an 
homogeneous element of $BP_*$ but also a homogeneous element of $V^A$.

In general, for any commutative graded ring $R$, any homogeneous element $r\in R$, and any graded $R$-module $M$,
the $R$-module $H^1_{(r)}(M) \cong (r^{-1}M)/M$ is $r$-power-torsion, since any homogeneous element $m\in r^{-1}M$ can be multiplied by
some power of $r$ to get an element of $M$, which represents zero in $H^1_{(r)}(M)$. Consequently $H^1_{(p)}(BP_*(X))$ is a
$p$-power-torsion graded $V^A$-module.
Proposition~\ref{basic properties of local cohomology 2} also
gives us that $H^0_{(p)}(BP_*(X))$ is a $p$-power-torsion graded $V^A$-module.
Now Lemma~\ref{power torsion preserved under extension} gives us that $BP_*(cX)$ is also a $p$-power-torsion graded $V^A$-module.
By Proposition~\ref{low degree gamma computation}, 
$\gamma(v_1) = \frac{p}{\pi}v_1^A$, so $BP_*(cX)$ being $p$-power-torsion implies that $BP_*(cX)$ is also $v_1$-power-torsion. 
So Lemma~\ref{acyclicity lemma} implies that
$L_{E(1)}cX$ is contractible. 

This is the initial step in an induction:
if $n$ is at least $2$ and we have already shown that $L_{E(n-1)}cX$ is contractible and $BP_*(cX)$ is $v_i$-power-torsion for all $i\leq n-1$,
then by Lemma~\ref{eventual divisibility of VA-modules} and
Lemma~\ref{modules over eventually div algs are eventually div}, 
$v_{n+1}$ acts with eventual division by $v_{n}$ on $BP_*(cX)$ modulo $I_{n}$.
Then Lemma~\ref{localization and eventual division} implies that $BP_*(cX)$ is $v_n$-power-torsion and that the Bousfield localization $L_{E(n)}cX$ is contractible,
completing the inductive step.
So $cX$ is acyclic with respect to all Johnson-Wilson theories, hence $cX$ is dissonant,
hence $X$ is an extension of a rational spectrum by a dissonant spectrum, as claimed.

Consequently, $L_{E(n)}cX$ is contractible for all $n\geq 0$, and $BP_*(cX)$ is $v_n$-power-torsion for all $n\geq 0$.
Consequently $cX$ is $E(n)$-acyclic, and smashing fiber sequence~\ref{fiber seq 60} with $E(n)$ yields the weak equivalence
$E(n)\smash X \stackrel{\simeq}{\longrightarrow} E(n)\smash LX$. Since $LX$ is a rational spectrum, it splits as a wedge of suspensions of copies of
$H\mathbb{Q}$, and hence $E(n)\smash LX$ splits as a wedge of suspensions of copies of $E(n)\smash HQ$.
Hence $E(n)_*(X)$ is a $\mathbb{Q}$-vector space, as claimed.

Smashing the fiber sequence~\ref{fiber seq 60} with the $E(n)$-local sphere $L_{E(n)}S$, and now using Theorem~\ref{smashing conj}, 
and the fact that $L_{E(n)}L_{E(0)}X \simeq L_{E(0)}X$, we get the homotopy fiber sequence
\[ L_{E(n)}cX \rightarrow L_{E(n)}X \rightarrow L_{E(0)}X .\]
We have shown that $L_{E(n)}cX$ is contractible, so the localization map $L_{E(n)}X \rightarrow L_{E(0)}X$ is an equivalence,
so $L_{E(n)}X$ is a rational spectrum as claimed. 

Furthermore, recall that, for all positive integers $n$, the notation $\mu_nX$ is used to denote the ``$n$th monochromatic layer of $X$,'' i.e., $\mu_nX$ sits in the
homotopy fiber sequence
\[ \mu_n X \rightarrow L_{E(n)}X \rightarrow L_{E(n-1)}X.\]
So $\mu_nX$ is contractible for all $n>0$. Recall also (from e.g.~\cite{MR1601906}) that there exists natural isomorphisms in the stable homotopy category 
$L_{K(n)}\circ \mu_n \simeq L_{K(n)}$ and $\mu_n\circ L_{K(n)} \simeq \mu_n$; consequently the contractibility of $\mu_nX$ implies the contractibility of 
$L_{K(n)}X$. So $X$ is $K(n)$-acyclic for all $n>0$, so $K(n)_*(X) \cong 0$ for all $n>0$, as claimed.
\end{proof}

\begin{corollary}\label{tot ram corollary for spaces}
Let $K/\mathbb{Q}_p$ be a finite, totally ramified extension of degree greater than one.
Let $X$ be a space such that $BP_*(X)$ is the underlying graded $BP_*$-module of a graded $V^A$-module.
Then $X$ is stably rational, that is,
its stable homotopy groups are all $\mathbb{Q}$-vector spaces.
\end{corollary}
\begin{proof}
By~\cite{MR1317578}, suspension spectra are harmonic,
hence $\Sigma^{\infty} X$ does not admit nontrivial maps from dissonant spectra.
By Theorem~\ref{main thm on tot ram nonexistence}
$\Sigma^{\infty} X$ is rational.
\end{proof}

\begin{corollary}\label{torsion-free and local coh}
Let $K/\mathbb{Q}_p$ be a finite, totally ramified extension of degree greater than one, and with ring of integers $A$.
Let $X$ be a spectrum such that $BP_*(X)$ is the underlying graded $BP_*$-module of a graded $V^A$-module,
and suppose that $BP_*(X)$ is torsion-free as an abelian group, i.e., $BP_*(X)$ is $p$-torsion-free.
Then the local cohomology module $H^1_{(p)}(BP_*(X))$ and the mod $p$ reduction $BP_*(X)/pBP_*(X)$ are both $v_n$-power-torsion $BP_*$-modules for all $n\geq 0$.
\end{corollary}
\begin{proof}
Since $BP_*(X)$ is assumed to be $p$-torsion-free, its maximal $p$-power-torsion submodule $\Gamma_{(p)}(BP_*(X))$ is trivial, so
from Propositions~\ref{basic properties of local cohomology 1} and~\ref{basic properties of local cohomology 2}
we get the short exact sequence of graded $V^A$-modules
\[ 0 \rightarrow BP_*(X) \rightarrow p^{-1}BP_*(X) \rightarrow H^1_{(p)}(BP_*(X)) \rightarrow 0,\]
and the localization map $BP_*(X) \rightarrow p^{-1}BP_*(X)$ coincides
with the localization map in homotopy $BP_*(X) \rightarrow BP_*(L_{E(0)}X)$.
Consequently $H^1_{(p)}(BP_*(X)) \cong BP_*(cX)$, the $BP$-homology
of the $E(0)$-acyclization $cX$ of $X$.
Now Theorem~\ref{main thm on tot ram nonexistence}
tells us that $BP_*(cX)$ is $v_n$-power-torsion for all $n\geq 0$.

Similarly, since $BP_*(X)$ is $p$-torsion-free, the $BP$-homology of $X$ smashed with the mod $p$ Moore spectrum is $BP_*(X)/pBP_*(X)$:
\[ BP_*(X \smash S/p) \cong BP_*(X)/pBP_*(X).\]
Of course the mod $p$ reduction of a graded $V^A$-module is still a graded $V^A$-module, so
$X\smash S/p$ is a spectrum with the property that $BP_*(X \smash S/p)$ is a graded $V^A$-module, hence by Theorem~\ref{main thm on tot ram nonexistence} $X\smash S/p$ is an extension of a rational spectrum by a dissonant spectrum. However, the homotopy groups of $X\smash S/p$ are all $p$-torsion,
hence rationally trivial, hence $X\smash S/p$ is $E(0)$-acyclic and does not map nontrivially to a rational spectrum; so $X \smash S/p$ is dissonant.
Now Proposition~\ref{dissonance lemma} gives us that $BP_*(X \smash S/p)$ is $v_n$-power-torsion for all $n\geq 0$.
\end{proof}

\begin{prop} {\bf (Hazewinkel.)}\label{hazewinkels formula}
Let $A$ be the ring of integers in a finite field extension of $\mathbb{Q}$. Let $\pi$ be a uniformizer for $A$ and let $\mathbb{F}_q$ be the residue field of $A$. 
Then the logarithm coefficients $\ell_1^A,\ell_2^A, \dots$
of the universal $A$-typical formal $A$-module satisfy the relation
\begin{eqnarray*} \pi \ell_h^A  & = & \ell^A_{h-1} (v_1^A)^{q^{h-1}} + \dots
 + \ell^A_1 (v^A)^q_{h-1} + v_h^A, \\
\ell_h^A & = & \sum_{i_1+ \dots + i_r = h} \pi^{-r} v_{i_1}^A(v_{i_2}^{q^{i_1}})
\dots (v_{i_r}^A)^{q^{i_1+\dots +i_{r-1}}},\end{eqnarray*}
where all $i_j$ are positive integers,
and where $v_1^A, v_2^A, \dots$ are Hazewinkel's generators
for $V^A$.
\end{prop}
\begin{proof}
This is formula 21.5.4 of \cite{MR506881}.
\end{proof}

\begin{prop} \label{action of v in tot ram case}
Let $K/\mathbb{Q}_p$ and $L/\mathbb{Q}_p$ be finite field extensions, and 
let $L/K$ be a totally ramified extension of degree $n>1$. Fix a positive integer $j$. Let $A,B$ be the rings of integers of $K,L$,
respectively,
and let $\mathbb{F}_q$ be their residue field. 
Let $\pi_A,\pi_B$ be uniformizers for $A,B$, respectively.
Let $\gamma: V^A \rightarrow V^B$ denote the ring map
classifying the underlying $A$-typical formal $A$-module of the universal $B$-typical formal $B$-module.
The least $h$ such that 
$\gamma(v_h^A)$ is nonzero modulo $(\pi_B, v_1^B, \dots ,v_{j-1}^B)$ is $h=jn$.
For $h=jn$ we have
\begin{eqnarray}
 \gamma(v_{jn}^A)
  \label{second equation for kappa} & \equiv & \frac{\pi_A}{\pi_B^n}(v_j^B)^{\frac{q^{jn}-1}{q^j-1}}\mod (\pi_B, v_1^B, \dots ,v_{j-1}^B).\end{eqnarray}
\end{prop}
\begin{proof} 
Write $I^B_j$ for the ideal $(\pi_B, v_1^B, \dots ,v_{j-1}^B)$ in $V^B$.
Clearly $\gamma(v_0^A) = \gamma(\pi_A)$ is zero modulo $I^B_j$.
That was the initial step in an induction. For the inductive step,
let $h$ be some positive integer, $h\leq jn$, and suppose that we have already shown that
$\gamma(v_{h^{\prime}}^A) \in I^B_j$ for all $h^{\prime}<h$. 
Using Proposition~\ref{hazewinkels formula} we have
\[ \gamma(v_h^A) = \frac{\pi_A}{\pi_B}v_h^B +  \sum_{i_1+\dots+i_r=h\mbox{\ and\ }r>1}\left(\frac{\pi_A}{\pi_B}v_{i_1}^B
\prod_{m=2}^r\frac{(v_{i_m}^B)^{q^{\sum_{s=1}^{m-1}i_s}}}{\pi_B} -
\gamma(v_{i_1}^A)\prod_{m=2}^r\frac{\gamma(v_{i_m}^A)^{q^{\sum_{s=1}^{m-1}i_s}}}{\pi_A}\right),\]
which is zero in 
$V^B/I^B_j$ unless $j\mid h$. Assume $j\mid h$; then we have
\begin{eqnarray*} \gamma(v_h^A) & \equiv & 
\frac{\pi_A}{\pi_B}v_j^B\prod_{m=2}^{h/j}\frac{(v_j^B)^{q^{(m-1)j}}}{\pi_B} \mod I^B_j \\
 & \equiv & \frac{\pi_A}{(\pi_B)^{h/j}}(v_j^B)^{\frac{q^h-1}{q^j-1}}\mod I^B_j \end{eqnarray*}
and $\nu_{\pi_B}\left(\frac{\pi_A}{(\pi_B)^{h/j}}\right) > 0$ if and only if $\frac{h}{j} < n$; since, in this case,
$r = \frac{h}{j}>1$ and $n>1$, we cannot have $\frac{h}{j}> n$, so the lowest $h$ such that 
$\gamma(v_h^A) \notin I^B_j$
is $h = jn$. This gives us equation~\ref{second equation for kappa}. 
\end{proof}

\begin{theorem}\label{torsion-free and local coh not fin pres}
Let $K/\mathbb{Q}_p$ be a finite, totally ramified extension of degree greater than one, and with ring of integers $A$.
Let $X$ be a spectrum such that $BP_*(X)$ is the underlying graded $BP_*$-module of a graded $V^A$-module.
Suppose that either $BP_*(X)$ is $p$-power-torsion, or that 
$BP_*(X)$ is torsion-free as an abelian group, i.e., $BP_*(X)$ is $p$-torsion-free.
Then $BP_*(X)$ is not finitely presented as a $BP_*$-module, and $BP_*(X)$ is not finitely presented as a $V^A$-module.
\end{theorem}
\begin{proof}
\begin{itemize}
\item Suppose first that $BP_*(X)$ is $p$-power-torsion. 
Then by Lemma~\ref{acyclicity lemma}, $L_{E(0)}X$ is contractible,
so $X$ is $E(0)$-acyclic and does not map nontrivially to a rational spectrum. So, by Theorem~\ref{main thm on tot ram nonexistence},
$X$ is dissonant, and $BP_*(X)$ is $v_n$-power-torsion for all $n\geq 0$. Now Lemma~\ref{finite presentation and power-torsion}
gives us that $BP_*(X)$ is not finitely presented as a $BP_*$-module. 

Similarly, if $[K :\mathbb{Q}_p] = d$, then $v_{dn}$ acts by a power of $v_n^A$ on $V^A/(\pi, v_1^A, \dots , v_{n-1}^A)$,
by Proposition~\ref{action of v in tot ram case}.
Consequently a finitely presented $V^A$-module cannot be $v_{dn}$-power-torsion for all $n\geq 0$.
So $BP_*(X)$ cannot be finitely presented as a $V^A$-module.
\item 
Now suppose that $BP_*(X)$ is $p$-torsion-free and finitely presented
as a $BP_*$-module.
Then $BP_*(X)/pBP_*(X)$ is also finitely presented as a $BP_*$-module, 
and by Corollary~\ref{torsion-free and local coh},
it is also $v_n$-power-torsion for all $n\geq 0$.
This contradicts Lemma~\ref{finite presentation and power-torsion}.

Suppose instead that $BP_*(X)$ is $p$-torsion-free and finitely presented
as a $V^A$-module.
Then $BP_*(X)/pBP_*(X)$ is also finitely presented as a $V^A$-module, 
and by Corollary~\ref{torsion-free and local coh},
it is also $v_n$-power-torsion for all $n\geq 0$.
Let $d = [K :\mathbb{Q}_p]$, and let
$\gamma: BP_* \rightarrow V^A$ be the map classifying the underlying $p$-typical formal group law of the universal $A$-typical
formal $A$-module.
By Proposition~\ref{action of v in tot ram case},
$\gamma(v_{dn})$ is congruent to a unit times a power of $v_n^A$ modulo $(\pi, v_1^A, \dots ,v_{n-1}^A)$. 
I claim that this implies that $BP_*(X)/pBP_*(X)$ is $v_n^A$-power-torsion for all $n\geq 0$.
The proof is as follows:
\begin{itemize}
\item Since $BP_*(X)/pBP_*(X)$ is $v_n$-power-torsion for all $n\geq 0$, the graded $V^A$-module
$\pi^i BP_*(X)/\pi^{i+1} BP_*(X)$ is also $v_n$-power-torsion for all $n\geq 0$
and all $i = 0, 1, \dots , d-1$,
by Lemma~\ref{quotient of power torsion is power torsion}. 
\item Since $\gamma(v_d)$ is congruent to a unit times $v_1^A$ modulo $\pi$ and since \linebreak $\pi^i BP_*(X)/\pi^{i+1} BP_*(X)$ is $v_d$-power-torsion 
for all $i = 0, 1, \dots , d-1$, we have that $\pi^i BP_*(X)/\pi^{i+1} BP_*(X)$ is also $v_1^A$-power torsion for all $i = 0, 1, \dots , d-1$.
\item That was the initial step in an induction, which proceeds as follows: suppose we have already shown that $\pi^i BP_*(X)/\pi^{i+1} BP_*(X)$ is $v_j^A$-power-torsion for all $j = 1, \dots, n-1$ and all $i = 0, 1, \dots , d-1$.
By Proposition~\ref{action of v in tot ram case}, 
\[ \gamma(v_{dn}) \equiv \frac{p}{\pi^d}(v_n^A)^{\frac{p^{dn}-1}{p^n-1}} + v_1^Ax_1
 + v_2^Ax_2 + \dots + v_{n-1}^Ax_{n-1}\mod \pi\]
for some $x_1, \dots ,x_{n-1}\in V^A$. 
For a given homogeneous $x\in\pi^iBP_*(X)/\pi^{i+1}BP_*(X)$, choose positive integers $\epsilon_1, \dots , \epsilon_{n-1}$ such that $(v_i^A)^{\epsilon_i}x = 0$, 
and choose a positive integer $\epsilon$ such that $v_{dn}^{\epsilon}x = 0$.
Let $\epsilon^{\prime}$ be any positive integer such that
$p^{\epsilon^{\prime}} >\epsilon$
and such that $p^{\epsilon^{\prime}} >\epsilon_j$ for all $j=1,\dots ,n-1$.
Then 
\begin{align*} 0 &= v_{dn}^{p^{\epsilon^{\prime}}} x \\
 &= \left( \frac{p}{\pi^d}(v_n^A)^{\frac{p^{dn}-1}{p^n-1}}\right)^{p^{\epsilon^{\prime}}}x + \sum_{j=1}^{n-1} \left( v_j^Ax_j \right)^{p^{\epsilon^{\prime}}}x \\
 &= \left(\frac{p}{\pi^d}\right)^{p^{\epsilon^{\prime}}} (v_n^A)^{p^{\epsilon^{\prime}}\frac{p^{dn}-1}{p^n-1}} x,\end{align*}
and since $\frac{p}{\pi^d}\in \mathbb{F}_p^{\times}$, we have that $(v_n^A)^{p^{\epsilon^{\prime}}\frac{p^{dn}-1}{p^n-1}} x = 0$.
Hence $\pi^iBP_*(X)/\pi^{i+1}BP_*(X)$ is $v_n^A$-power-torsion for all $i=0, \dots ,d-1$, completing the inductive step. Hence $\pi^iBP_*(X)/\pi^{i+1}BP_*(X)$ is $v_n^A$-power-torsion for all $i=0, \dots ,d-1$ and for all $n\geq 0$.
\item Now repeated application of Lemma~\ref{power torsion preserved under extension} to the finite tower of extensions of graded $V^A$-modules
\[\xymatrix{ 
 \pi^{d-1}BP_*(X)/\pi^dBP_*(X) \ar@{^{(}->}[r] & BP_*(X)/pBP_*(X) \ar@{->>}[d] \\
 \pi^{d-2}BP_*(X)/\pi^{d-1}BP_*(X) \ar@{^{(}->}[r] & BP_*(X)/\pi^{d-2}BP_*(X) \ar@{->>}[d] \\
 & \vdots \\
 \pi BP_*(X)/\pi^2 BP_*(X) \ar@{^{(}->}[r] & BP_*(X)/\pi^2 BP_*(X) \ar@{->>}[d] \\
 BP_*(X)/\pi BP_*(X) \ar[r]^{\cong} & BP_*(X)/\pi BP_*(X)}\]
shows that $BP_*(X)/pBP_*(X)$ is $v_n^A$-power-torsion for all $n\geq 0$.
\end{itemize}
Again Lemma~\ref{finite presentation and power-torsion} then gives us a contradiction: $BP_*(X)$ cannot be finitely presented as a $V^A$-module.
\end{itemize}
\end{proof}

\begin{corollary}\label{nonrealizability of VA}
Let $K/\mathbb{Q}_p$ be a finite field extension of degree greater than one, and with ring of integers $A$.
Then there does not exist a spectrum $X$ such that $BP_*(X) \cong V^A$
as graded $BP_*$-modules.
\end{corollary}
\begin{proof}
Immediate from Corollary~\ref{unramified VA nonrealizability} together with Theorem~\ref{torsion-free and local coh not fin pres}.
\end{proof}

\section{Conclusions.}

\begin{theorem}{\bf (Nonexistence of Ravenel's algebraic spheres, local case.)}
Let $K$ be a finite extension of $\mathbb{Q}_p$ of degree $>1$, with ring of integers $A$.
Then there does not exist a spectrum $X$ such that $MU_*(X) \cong L^A$.
\end{theorem}
\begin{proof}
This is a proof by contrapositive: suppose that a spectrum $X$ exists so that $MU_*(X)\cong L^A$. Let $q$ denote the cardinality of the residue
field of $A$.
We have an isomorphism of graded $V^{A}$-modules
\[ L^{A} \cong \coprod_{0\leq n\neq q^k-1}  \Sigma^{2n} V^{A}\]
given by Cartier typification (see 21.7.17 of~\cite{MR506881} for this).
Now, since $MU_*(X)$ is already $p$-local, 
we have the isomorphism of graded $MU_*$-modules
\begin{align*}
 MU_*(X) & \cong  \coprod_{0\leq n\neq p^k-1} \Sigma^{2n} BP_*(X) .\end{align*}
Let $Y$ be the coproduct spectrum $Y = \coprod_{0\leq n\neq p^k-1} \Sigma^{2n} X$,
so that we have isomorphisms of graded $BP_*$-modules
\begin{align*}
 BP_*(Y) 
  & \cong  \coprod_{0\leq n\neq p^k-1} \Sigma^{2n} BP_*(X) \\
  & \cong  MU_*(X) \\
  & \cong  L^A \\
  & \cong  \coprod_{0\leq n\neq q^k-1} \Sigma^{2n} V^A .\end{align*}
So the graded $V^A$-module $\coprod_{0\leq n\neq q^k-1} \Sigma^{2n} V^A$ is the $BP$-homology of the spectrum $Y$. But
$\coprod_{0\leq n\neq q^k-1} \Sigma^{2n} V^A$ is a free $V^A$-module.
\begin{itemize}
\item
By Theorem~\ref{unramified dissonance thm}, if $K/\mathbb{Q}_p$ is not totally ramified, then $Y$ must be dissonant, and in particular, $BP_*(Y)$ must be $p$-power-torsion; but this contradicts $BP_*(Y)$ being a free $V^A$-module, since $V^A$ is $p$-torsion free.
\item
On the other hand, if $K/\mathbb{Q}_p$ is totally ramified, then by 
Theorem~\ref{main thm on tot ram nonexistence}, the fiber $cY$ of the rationalization map $Y \rightarrow LY$ is dissonant. Since $BP_*(Y)$ is free and $BP_*(LY) \cong p^{-1}BP_*(Y)$, we have a commutative diagram of graded $BP_*$-modules with exact rows
\[\xymatrix{
 0 \ar[r]\ar[d] & BP_*(Y) \ar[r]\ar[d]^{\cong} & BP_*(LY) \ar[r] \ar[d]^{\cong} & BP_*(\Sigma cY) \ar[r]\ar[d]^{\cong} &  0\ar[d] \\
 0 \ar[r] & BP_*(Y) \ar[r] & p^{-1}BP_*(Y) \ar[r] & H^1_{(p)}(BP_*(Y)) \ar[r] & 0 }\]
and, again since $BP_*(Y)$ is a free $V^A$-module, 
$BP_*(\Sigma cY)$ is isomorphic to a direct sum of copies of $H^1_{(p)}(V^A) \cong p^{-1}V^A/V^A$.
Let $n = [ K : \mathbb{Q}_p]$. Then Proposition~\ref{action of v in tot ram case} implies that $v_{n}$ acts on $V^A$ as $\frac{p}{\pi^n}(v_1^A)^{\frac{p^n-1}{p-1}}$ modulo $\pi$. Since $\frac{p}{\pi^n}\in \mathbb{F}_p^{\times}$, we now have that $p^{-1}V^A/V^A$ is $v_n$-torsion-free, consequently $BP_*(\Sigma cY)$ is $v_n$-torsion-free. But by Proposition~\ref{dissonance lemma}, $BP_*(\Sigma cY)$ then cannot be dissonant, a contradiction.
\end{itemize}
So the spectrum $X$ such that $MU_*(X)\cong L^A$ cannot exist.
\end{proof}

The following lemma is classical (and stronger results along these lines are probably known to those who know more number theory than I do!):
\begin{lemma}\label{chebotarev density lemma}
Let $K/\mathbb{Q}$ be a finite extension of degree $>1$, with ring of integers $A$. Then there exist infinitely many prime numbers 
$p$ which neither ramify nor split completely in $A$.

If we furthermore assume that either $K/\mathbb{Q}$ is cyclic of prime power order or that $K/\mathbb{Q}$ is Galois, then there exists some prime number $p$ which neither ramifies nor splits completely in $A$,
and such that each prime of $A$ over $p$ is of the same degree. (That is, there exists some positive integer $n$ such that every prime of $A$ over $p$ is of degree $d$.)
\end{lemma}
\begin{proof}
See Corollary~5 of Proposition~7.16 of~\cite{MR2078267} for the fact 
that, in any finite extension of $\mathbb{Q}$, there are infinitely many primes of $\mathbb{Z}$ which do not split completely. 
(On the other hand, unless $K/\mathbb{Q}$ is cyclic, then there are only finitely many primes of $\mathbb{Z}$ which do not split {\em at all} in $A$; see Exercise~VI.3.1 of~\cite{MR1697859}.) Only finitely many such primes can ramify in $A$; the rest neither ramify nor split completely.

If $K/\mathbb{Q}$ is cyclic of prime power order, then there exist infinitely many primes of $\mathbb{Z}$ which do not split {\em at all} in $A$; see Corollary~VI.3.7 of~\cite{MR1697859}. Only finitely many such primes ramify in $A$, and the rest neither ramify nor split at all; any one of those primes will satisfy the claim made in the statement of the lemma.

On the other hand, if $K/\mathbb{Q}$ is Galois, then the Galois group $\Gal(K/\mathbb{Q})$ acts transitively on the set of primes of $A$ over $p$. Hence every prime of $A$ over $p$ has the same degree, and any prime $p$ which neither ramifies nor splits completely in $A$ will satisfy the claim made in the statement of the lemma.
\end{proof}

Lemma~\ref{completion lemma} is another easy one and which is definitely not new, but for which I do not know a reference in the literature.
\begin{lemma}\label{completion lemma}
Let $A$ be a Noetherian commutative ring, 
and let $\mathfrak{m}_1, \dots ,\mathfrak{m}_n$ be a finite set of maximal ideals of $A$.
Let $\mathfrak{m}$ denote the product of these ideals,
\[ \mathfrak{m} = \mathfrak{m}_1\cdot \mathfrak{m}_2\cdot \dots \cdot \mathfrak{m}_n\subseteq A.\]
Suppose that the ring $A/\mathfrak{m}$ is Artinian.
Then there is an isomorphism of $A$-algebras 
\[ \hat{A}_{\mathfrak{m}} \cong \prod_{i=1}^n \hat{A}_{\mathfrak{m}_i}\]
between the $\mathfrak{m}$-adic completion of $A$ and the product of the $\mathfrak{m}_i$-adic completions of $A$ for all $i$
between $1$ and $n$.

Furthermore, for any finitely generated $A$-module $M$, the $\mathfrak{m}$-adic completion of $M$ is isomorphic to the
product of the $\mathfrak{m}_i$-adic completions of $M$ for all $i$ between $1$ and $n$, i.e.,
\[ \hat{M}_{\mathfrak{m}} \cong \prod_{i=1}^n \hat{M}_{\mathfrak{m}_i}.\]
This isomorphism is natural in the choice of finitely generated $A$-module $M$.
\end{lemma}
\begin{proof}
Since $M$ is finitely generated and $A$ is Noetherian, we have an isomorphism
\begin{equation}\label{smashing iso 2} \hat{M}_{I} \cong \hat{A}_I \otimes_A M\end{equation}
for any choice of ideal $I$ in $A$,
and isomorphism~\ref{smashing iso 2} is natural in the choice of finitely generated $A$-module $M$. 
(This is classical; see e.g. Proposition~10.13 of~\cite{MR0242802}, or consult any number of other algebra textbooks.)

Now, for any positive integer $i$, the kernel of the $A$-algebra homomorphism
\[ A/\mathfrak{m}^j \rightarrow A/\mathfrak{m} \]
is $\mathfrak{m}/\mathfrak{m}^j$, a nilpotent ideal in $A/\mathfrak{m}_j$;
hence $A/\mathfrak{m}^j$ is a nilpotent extension of an Artinian ring for all $j\in\mathbb{N}$.
Hence $A/\mathfrak{m}^j$ is Artinian for all $j\in\mathbb{N}$.

Now every Artinian ring splits as the Cartesian product of its localizations at its maximal ideals (this is Theorem~8.7 in \cite{MR0242802}, for example),
so we have an isomorphism of rings $A/\mathfrak{m}^j \cong \prod_{i=1}^n A/\mathfrak{m}_i^j$,
and on taking limits, isomorphisms of rings
\begin{align*} 
 \hat{A}_{\mathfrak{m}} 
  & \cong  \lim_{j\rightarrow\infty} A/\mathfrak{m}^j \\
  & \cong  \lim_{j\rightarrow\infty} \left( \prod_{i=1}^n A/\mathfrak{m}^j_i\right) \\
  & \cong  \prod_{i=1}^n\lim_{j\rightarrow\infty} A/\mathfrak{m}^j_i \\
  & \cong  \prod_{i=1}^n \hat{A}_{\mathfrak{m}_i},\end{align*}
hence isomorphisms of $A$-modules
\begin{align*}
 \hat{M}_{\mathfrak{m}} 
  & \cong  \hat{A}_{\mathfrak{m}}\otimes_A M \\
  & \cong  \left( \prod_{i=1}^n\hat{A}_{\mathfrak{m}_i}\right)\otimes_A M \\
  & \cong  \left( \coprod_{i=1}^n\hat{A}_{\mathfrak{m}_i}\right)\otimes_A M \\
  & \cong  \coprod_{i=1}^n\left( \hat{A}_{\mathfrak{m}_i}\otimes_A M\right) \\
  & \cong  \coprod_{i=1}^n\left( \hat{M}_{\mathfrak{m}_i}\right) \\
  & \cong  \prod_{i=1}^n\left( \hat{M}_{\mathfrak{m}_i}\right), \end{align*}
all natural in $M$.
(The switch from $\times$ to $\oplus$ is because coproducts commute with the tensor product of modules;
and the product in question happens to be a finite product of modules, hence it is isomorphic to a finite coproduct of modules.)
\end{proof}

\begin{lemma}\label{completion commutes with V and with L}
Let $A$ be a discrete valuation ring with maximal ideal $\mathfrak{m}$ and with finite residue field. Let $g,h$ be the natural maps of commutative graded rings
\[ (L^A)\otimes_A \hat{A}_{\mathfrak{m}} \stackrel{h}{\longrightarrow} (L^A)^{\hat{}}_{\mathfrak{m}} \stackrel{g}{\longrightarrow} L^{\hat{A}_{\mathfrak{m}}}, \]
where $h$ is the usual completion comparison map (as in Proposition~10.13 of~\cite{MR0242802}), and $g$ is the map classifying the
underlying formal $A$-module of the universal formal $\hat{A}_{\mathfrak{m}}$-module.
Let $g^{\prime},h^{\prime}$ be the natural maps of commutative graded rings
\[ (V^A)\otimes_A \hat{A}_{\mathfrak{m}} \stackrel{h^{\prime}}{\longrightarrow} (V^A)^{\hat{}}_{\mathfrak{m}} \stackrel{g^{\prime}}{\longrightarrow} V^{\hat{A}_{\mathfrak{m}}}, \]
where $h^{\prime}$ is the usual completion comparison map, and $g^{\prime}$ is the map classifying the
underlying $A$-typical formal $A$-module of the universal $\hat{A}_{\mathfrak{m}}$-typical formal $\hat{A}_{\mathfrak{m}}$-module.

Then $g,h,g^{\prime}$, and $h^{\prime}$ are all isomorphisms.
\end{lemma}
\begin{proof}
For any commutative ring $A$ and any integer $n>2$, Drinfeld proved in~\cite{MR0384707} that the module of indecomposables $L^A_{n-1}/D^A_{n-1}$ of $L^A$ in grading degree $2(n-1)$ is the $A$-module with one generator $c_a$ for each $a\in A$ and one additional generator $d$, and subject to the relations
\begin{align*} d(a-a^n) &= \nu(n)c_a \mbox{\ \ for\ all\ } a\in A \\
 c_{a+b}-c_a-c_b &= d\frac{a^n + b^n - (a+b)^n}{\nu(n)} \mbox{\ \ for\ all\ } a,b\in A \\
 ac_b + b^nc_a &= c_{ab} \mbox{\ \ for\ all\ } a,b\in A .\end{align*}
Here $\nu(n)$ is defined to be $p$ if $n$ is a power of a prime number $p$, and $\nu(n) = 1$ if $n$ is not a prime power.
Hazewinkel then proved (this is Proposition~21.3.1 of~\cite{MR506881}) that, if $A$ is a discrete valuation ring with finite residue field, then
$L^A_{n-1}/D^A_{n-1}$ is isomorphic to a free $A$-module on one generator. Clearly this is true when $\nu(n)= 1$, since then $L^A_{n-1}/D^A_{n-1}$ is generated by $d$. Suppose instead that $\nu(n)\neq 1$.
Let $I^A$ denote the ideal in $A$ generated by all elements of the form 
$a-a^n, a\in A$.
Then, as Hazewinkel observes, there are two possibilities:
\begin{description}
\item[If $I^A = A$] Then there exists an $a\in A$ such that $a-a^n$ is a unit in $A$,
and consequently $c_a$ generates $L^A_{n-1}/D^A_{n-1}$.
\item[If $I^A \neq A$] Then, writing $\pi$ for a generator of $\mathfrak{m}$, the element $c_{\pi}$ generates $L^A_{n-1}/D^A_{n-1}$. Hazewinkel observes that this case only occurs when $A/\mathfrak{m}$ is finite and $n$ is a power of the characteristic of $A/\mathfrak{m}$.
\end{description}

Since the completion map $f: A\rightarrow \hat{A}_{\mathfrak{m}}$ is injective,
if $I^A = A$, then there exists some finite collection of elements
$a_1,b_1, a_2,b_2, \dots ,a_m,b_m\in A$ such that
$\sum_{i=1}^m b_m(a_m - a_m^n)  = 1$,
and consequently 
$\sum_{i=1}^m f(b_m)(f(a_m) - f(a_m)^n)  = 1$ in $\hat{A}_{\mathfrak{m}}$.
So $I^{\hat{A}_{\mathfrak{m}}} = \hat{A}_{\mathfrak{m}}$.
Choose an element $a\in A$ such that $a-a^n\in A^{\times}$,
and now the
map 
\[ A\{ c_a\} \stackrel{\cong}{\longrightarrow} L^A_{n-1}/D^A_{n-1} 
 \rightarrow L^{\hat{A}_{\mathfrak{m}}}_{n-1}/D^{\hat{A}_{\mathfrak{m}}}_{n-1} 
 \stackrel{\cong}{\longrightarrow} \hat{A}_{\mathfrak{m}}\{ c_{f(a)}\} \]
is simply the $\mathfrak{m}$-adic completion map.

On the other hand, if $I^A \neq A$, I claim that $I^{\hat{A}_{\mathfrak{m}}} \neq \hat{A}_{\mathfrak{m}}$ as well. We prove this by contrapositive: suppose that
$I^A \neq A$ but $I^{\hat{A}_{\mathfrak{m}}} = \hat{A}_{\mathfrak{m}}$.
Then there exists an element $a\in \hat{A}_{\mathfrak{m}}$ such that
$a-a^n \in \hat{A}_{\mathfrak{m}}^{\times}$, and consequently
the residue $\overline{a}\in \hat{A}_{\mathfrak{m}}/\mathfrak{m}\hat{A}_{\mathfrak{m}}$
also has the property that $\overline{a} - \overline{a}^n$ is a unit, and in particular,
$\overline{a}$ is nonzero. 
Since 
\begin{equation}\label{residue field isomorphism} \hat{A}_{\mathfrak{m}}/\mathfrak{m}\hat{A}_{\mathfrak{m}} \cong A/\mathfrak{m},\end{equation} we can choose an element $\tilde{a}\in A$ such that the residue of $\tilde{a}$ in $A/\mathfrak{m}$ is (under the isomorphism~\ref{residue field isomorphism}) 
$\overline{a}$. Now $\tilde{a}$ is an element of $A$ such that
$\tilde{a} - \tilde{a}^n$ is nonzero modulo $\mathfrak{m}$,
hence $I^A = A$ since every ideal in the discrete valuation ring $A$ is a power of $\mathfrak{m}$, so the only ideal containing an element
which is nonzero modulo $\mathfrak{m}$ is $A$ itself and so $\tilde{a} - \tilde{a}^n$ is a unit in $A$.
But this contradicts our assumption that $I^A \neq A$.

So if $I^A\neq A$, then $I^{\hat{A}_{\mathfrak{m}}} \neq \hat{A}_{\mathfrak{m}}$.
Writing $\pi$ for a generator of the maximal ideal $\mathfrak{m}$ of $A$,
we now have that the composite
\[ A\{ c_{\pi}\} \stackrel{\cong}{\longrightarrow} L^A_{n-1}/D^A_{n-1} 
 \rightarrow L^{\hat{A}_{\mathfrak{m}}}_{n-1}/D^{\hat{A}_{\mathfrak{m}}}_{n-1} 
 \stackrel{\cong}{\longrightarrow} \hat{A}_{\mathfrak{m}}\{ c_{\pi}\} \]
is simply the $\mathfrak{m}$-adic completion map.

In either case, the map $L^A_{n-1}/D^A_{n-1}  \rightarrow L^{\hat{A}_{\mathfrak{m}}}_{n-1}/D^{\hat{A}_{\mathfrak{m}}}_{n-1}$
is simply (up to isomorphism) $\mathfrak{m}$-adic completion.
Hazewinkel shows that, for any discrete valuation ring $A$,
$L^A$ is simply the symmetric $A$-algebra on 
the direct sum $\coprod_{n\geq 2} L^A_{n-1}/D^A_{n-1}$;
and now,
\begin{itemize}
\item since $A$ and $\hat{A}_{\mathfrak{m}}$ are both discrete valuation rings,
\item since the map $\coprod_{n\geq 2} L^A_{n-1}/D^A_{n-1}
\rightarrow \coprod_{n\geq 2} L^{\hat{A}_{\mathfrak{m}}}_{n-1}/D^{\hat{A}_{\mathfrak{m}}}_{n-1}$
is simply $\mathfrak{m}$-adic completion,
\item since $A$ is Noetherian and $\coprod_{n\geq 2} L^A_{n-1}/D^A_{n-1}$ is a finitely-generated $A$-module in each grading degree, and consequently
$\mathfrak{m}$-adic completion coincides with tensoring over $A$ with $\hat{A}_{\mathfrak{m}}$, 
\item and since taking symmetric algebras commutes with
tensoring over $A$ with $\hat{A}_{\mathfrak{m}}$, 
\end{itemize}
we conclude that the map $L^A \rightarrow L^{\hat{A}_{\mathfrak{m}}}$
is simply, up to isomorphism, $\mathfrak{m}$-adic completion, and 
it also coincides with the base-change map from $A$ to $\hat{A}_{\mathfrak{m}}$. The maps $g,h$ in the statement of the lemma are hence isomorphisms.

Since the map of discrete valuation rings $A\rightarrow \hat{A}_{\mathfrak{m}}$ sends a uniformizer to a uniformizer and 
induces an isomorphism on residue fields, the splitting of $L^A$ as an infinite direct sum of 
copies of $V^A$ is compatible with the maps from $L^A$ to $L^{\hat{A}_{\mathfrak{m}}}$ and from $V^A$ to $V^{\hat{A}_{\mathfrak{m}}}$; see 
section~21.5 of~\cite{MR506881}.
\end{proof}

We also need to use an easy result from the paper~\cite{cmah1}:
\begin{prop}\label{LA is torsion-free}
Let $A$ be the ring of integers in a finite field extension $K/\mathbb{Q}$. Then, for each integer $n$, the grading degree $n$ summand $L^A_n$ in the graded ring $L^A$ is a 
finitely generated $A$-module.
\end{prop}
\begin{proof}
See~\cite{cmah1}. (But, for the sake of self-containedness: one can easily prove this proposition from first principles by simply remembering that, to specify a formal $A$-module $n+1$-bud $(F(X,Y),\rho(X))$, one only needs to specify finitely many coefficients in the truncated power series $F(X,Y)$ and $\rho(X)$.)
\end{proof}

\begin{theorem}{\bf (Nonexistence of Ravenel's algebraic spheres, global case.)}\label{nonexistence of alg spheres global case}
Let $K$ be a finite extension of $\mathbb{Q}$ of degree $>1$, with ring of integers $A$. 
Suppose that there exists a prime number $p$ which neither ramifies nor splits completely in $A$, and 
such that each prime $\mathfrak{m}$ of $A$ over $p$ is of the same degree. (By Lemma~\ref{chebotarev density lemma}, this condition is satisfied if $K/\mathbb{Q}$ is cyclic of prime power order, and it is also satisfied if $K/\mathbb{Q}$ is Galois.)

Then there does not exist a spectrum $X$ such that $MU_*(X) \cong L^A$.
\end{theorem}
\begin{proof}
This is a proof by contrapositive: suppose that a spectrum $X$ exists so that $MU_*(X)\cong L^A$, and 
choose a prime number $p$ such that $p$ neither ramifies nor completely splits in $A$ and such that every prime of $A$ over $p$ is of the same degree.
It is classical that the homotopy limit $\holim_n S/p^n$ of the mod $p^n$ Moore spectra has its homotopy groups $\pi_*(\holim_n S/p^n)$
isomorphic, as a $\pi_*(S)$-module, to the $p$-adic completion
$(\pi_*(S))^{\hat{}}_p$ of $\pi_*(S)$.
Since $\pi_*(S)$ is finitely generated in each grading degree,
$(\pi_*(S))^{\hat{}}_p \cong \pi_*(S) \otimes_{\mathbb{Z}} \hat{\mathbb{Z}}_p$,
and since $\hat{\mathbb{Z}}_p$ is flat as a $\mathbb{Z}$-module,
the ring $(\pi_*(S))^{\hat{}}_p$ is flat as an $\pi_*(S)$-module.
Hence the $E^2$-term
\[ \Tor^{\pi_*(S)}_{*,*}\left(MU_*(X), \pi_*(\holim_n S/p^n)\right) 
  \]
of the 
K\"{u}nneth spectral sequence 
computing $\pi_*\left( MU \smash X \smash \holim_n S/p^n \right)$
collapses on to the $\Tor^{\pi_*(S)}_0$-line, hence
\begin{align}
\nonumber  MU_*\left( X \smash \holim_n S/p^n \right) &\cong
 MU_*(X) \otimes_{\pi_*(S)} \pi_*(S)^{\hat{}}_p  \\
 \nonumber &\cong  L^A\otimes_{\mathbb{Z}} \hat{\mathbb{Z}}_p \\
 \label{iso 5607} &\cong (L^A)^{\hat{}}_p ,\end{align}
with isomorphism~\ref{iso 5607} due to $L^A$ being finitely generated
in each grading degree by Proposition~\ref{LA is torsion-free}.
(Recall that, for a finitely generated module $M$ over a Noetherian ring
$R$, if $I$ is an ideal in $R$ then the natural map
$M\otimes_A \hat{A}_I \rightarrow \hat{M}_I$
is an isomorphism; see e.g. Proposition~10.13 of~\cite{MR0242802}.)

Now write $\mathfrak{p}_1, \mathfrak{p}_2, \dots ,\mathfrak{p}_m$ for the set of primes of $A$ over $p$.
By Lemma~\ref{completion lemma} and by the finite generation result in Proposition~\ref{LA is torsion-free}, $(L^A)^{\hat{}}_{p} \cong \prod_{i=1}^n (L^A)^{\hat{}}_{\mathfrak{p}_i}$.

Now, for each $i$,
let $q_i$ denote the cardinality of the residue
field of $\hat{A}_{\mathfrak{p}_i}$.
We have an isomorphism of graded $V^{\hat{A}_{\mathfrak{p}_i}}$-modules
\[ L^{\hat{A}_{\mathfrak{p}_i}} \cong \coprod_{0\leq n\neq q_i^k-1}  \Sigma^{2n} V^{\hat{A}_{\mathfrak{p}_i}}\]
given by Cartier typification (see 21.7.17 of~\cite{MR506881} for this).

Finally, putting all these isomorphisms together, we have a string of isomorphisms:
\begin{align}
\nonumber  MU_*\left( X \smash \holim_b S/p^b \right) &\cong
 (L^A)^{\hat{}}_p  \\
\nonumber  & \cong  \prod_{i=1}^m (L^{A})^{\hat{}}_{\mathfrak{p}_i} \\
\label{iso 7106}  & \cong  \prod_{i=1}^m L^{\hat{A}_{\mathfrak{p}_i}} \\
\label{iso 7107}  & \cong  \prod_{i=1}^m \left( \coprod_{0\leq a\neq q_i^j-1} \Sigma^{2a} V^{\hat{A}_{\mathfrak{p}_i}}\right) ,\end{align}
with isomorphism~\ref{iso 7106} due to 
Lemma~\ref{completion commutes with V and with L}.

Now by assumption, 
each of the primes $\mathfrak{p}_i$ in $A$ has the same residue degree;
and since $p$ was assumed to not ramify, this means that the extension of local fields
$K(\hat{A}_{\mathfrak{p}_i})/\mathbb{Q}_p$ is unramified of some residue degree $f_p$, and this residue degree $f_p$ does not depend on $i$. (I am writing $K(A)$ for the field of fractions of a ring $A$.)
But up to isomorphism, there is only one unramified extension of $\mathbb{Q}_p$ of a given residue degree $f_p$.
Consequently neither $K(\hat{A}_{\mathfrak{p}_i})$ nor its ring of integers $\hat{A}_{\mathfrak{p}_i}$ depends, up to isomorphism, on $i$.
So from isomorphism~\ref{iso 7107} we get the isomorphism
\begin{align}
\nonumber MU_*\left( X \smash \holim_b S/p^b \right)
 &\cong  \coprod_{0\leq a\neq q^j-1} \Sigma^{2a} \left(V^{\hat{A}_{\mathfrak{p}}}\right)^{\times n} \end{align}
where $\mathfrak{p}$ is any of the primes of $A$ over $p$,
and $q$ is the cardinality of the residue field of $\hat{A}_{\mathfrak{p}}$.

Now we have the usual splitting of $p$-local (hence also $p$-complete) complex bordism as a sum of suspensions of $BP$-homology, 
\begin{align*} 
 (MU_*(X))^{\hat{}}_p 
  & \cong  \coprod_{0\leq n\neq p^k-1} \Sigma^{2n} (BP_*(X))^{\hat{}}_p \\
  & \cong  BP_*\left(\left(\coprod_{0\leq n\neq p^k-1} \Sigma^{2n} X\right)\smash \holim_b S/p^b\right) ,\end{align*}
so now we have isomorphisms
\begin{align*}
 BP_*\left(\left(\coprod_{0\leq n\neq p^k-1} \Sigma^{2n} X\right)\smash \holim_b S/p^b\right) 
  & \cong  (MU_*(X))^{\hat{}}_p \\
  & \cong  (L^A)_p^{\hat{}} \\
  & \cong   \left(\coprod_{0\leq n\neq q^k - 1}\Sigma^{2n} V^{\hat{A}_{\mathfrak{p}}}\right)^{\times m} ,\end{align*}
and
since we assumed that $p$ does not completely split in $A$, the residue degree $f_p$ is greater than one.
So Theorem~\ref{unramified dissonance thm} applies:
$\left(\coprod_{0\leq n\neq q^k - 1}\Sigma^{2n} V^{\hat{A}_{\mathfrak{p}}}\right)^{\times f_p}$ must be a $v_n$-power-torsion $BP_*$-module for all 
$n\geq 0$, since
the $V^{\hat{A}_{\mathfrak{p}}}$-module $\left(\coprod_{0\leq n\neq q^k - 1}\Sigma^{2n} V^{\hat{A}_{\mathfrak{p}}}\right)^{\times f_p}$
is the $BP$-homology of a spectrum, namely, the spectrum
$\left(\coprod_{0\leq n\neq p^k-1} \Sigma^{2n} X\right)\smash \holim_b S/p^b$.
But 
$\left(\coprod_{0\leq n\neq q^k - 1}\Sigma^{2n} V^{\hat{A}_{\mathfrak{p}}}\right)^{\times f_p}$ is a free $V^{\hat{A}_{\mathfrak{p}}}$-module,
hence is not $p$-power-torsion, a contradiction. 

Hence there cannot exist $X$ such that $MU_*(X) \cong L^A$.
\end{proof}

\begin{corollary}\label{main conclusion cor}
Except in the trivial case where $A = \hat{\mathbb{Z}}_p$, Ravenel's local and $p$-typical algebraic extensions of the sphere spectrum do not exist. In the case of an extension which is cyclic of prime power order or a Galois extension, Ravenel's global algebraic extensions of the sphere spectrum do not exist, except in the trivial case where $A = \mathbb{Z}$.

That is, if $K/\mathbb{Q}$ is a finite extension of degree $>1$, which is either cyclic of prime power order or Galois, and which has ring of integers $A$, there does not exist a spectrum $S^A$
such that $MU_*(S^A) \cong L^A$ as graded $MU_*$-modules.
If $K/\mathbb{Q}_p$ is a finite extension of degree $>1$ with ring of integers $A$, then there does not exist a spectrum $S^A$
such that $MU_*(S^A) \cong L^A$ as graded $MU_*$-modules,
and there does not exist a spectrum $S^A$
such that $BP_*(S^A) \cong V^A$ as graded $BP_*$-modules.
\end{corollary}

\section{Appendix on basics of local cohomology.}\label{local cohomology appendix}

In this short appendix I present some of the most basic definitions and
ideas of local cohomology.
The reader who wants something more substantial about
this subject, or to read proofs of the results cited in this appendix, can consult the textbook of Brodmann and Sharp,
\cite{MR3014449}.

First, recall that an ideal $I$ in a commutative graded ring $R$ 
is called {\em homogeneous} if there exists a generating set for $I$
consisting of homogeneous elements.
\begin{definition}
Let $R$ be a commutative graded ring, and let $I$ be a homogeneous ideal of $R$. Let $\grMod(R)$ denote the category of graded $R$-modules
and grading-preserving $R$-module homomorphisms.
We define two functors
\begin{align*} D_I, \Gamma_I: \grMod(R) & \rightarrow  \grMod(R)\end{align*}
as follows:
\begin{align*}
 \Gamma_I(M) & =  \coprod_{m\in\mathbb{Z}}\colim_{n\rightarrow \infty}\hom_{\grMod(R)}(\Sigma^m R/I^n, M) \\
 D_I(M) & =  \coprod_{m\in\mathbb{Z}}\colim_{n\rightarrow \infty}\hom_{\grMod(R)}(\Sigma^m I^n, M), \end{align*}
with natural transformations
\begin{equation}\label{seq of nat transformations} \Gamma_I \rightarrow \id \rightarrow D_I \end{equation}
induced by applying $\coprod_{m\in\mathbb{Z}}\colim_{n\rightarrow \infty}\hom_{\grMod(R)}(\Sigma^m -, M)$ to the short exact sequence of graded $R$-modules
\begin{equation*}\label{} 0 \rightarrow I^n \rightarrow R \rightarrow R/I^n \rightarrow 0.\end{equation*}

The right-derived functors $R^*\Gamma_I$ of the functor $\Gamma_I$ are called {\em local cohomology at $I$}, and written
\[ H^n_I(M) = (R^n(\Gamma_I))(M).\]
\end{definition}

\begin{prop}\label{basic properties of local cohomology 1}
Let $R$ be a commutative graded ring, and let $I$ be a homogeneous ideal of $R$. 
The natural transformations~\ref{seq of nat transformations}
induce an exact sequence of $R$-modules
\begin{equation}\label{natural seq of transformations 1} 0 \rightarrow \Gamma_I(M) \rightarrow M \rightarrow D_I(M)
 \rightarrow H^1_I(M) \rightarrow 0,\end{equation}
natural in $M$, 
and an isomorphism 
\[ H^n_I(M) \cong (R^{n-1}(D_I))(M) \]
for all $n\geq 2$, also natural in $M$.
\end{prop}

\begin{prop}\label{basic properties of local cohomology 2}
Let $R$ be a commutative graded ring, and 
let $r\in R$ be a homogeneous element. Let $(r)$ be the principal
ideal generated by $r$.
\begin{itemize}
\item Then $D_{(r)}(M)$ is naturally isomorphic to $r^{-1}M$, and the natural 
transformation $\id \rightarrow D_{(r)}$ from~\ref{seq of nat transformations}
is naturally isomorphic to the standard localization map $M \rightarrow r^{-1}M$.
\item
Furthermore, $\Gamma_{(r)}(M)$ is naturally isomorphic to the
sub-$R$-module of $M$ consisting of the $r$-power-torsion elements,
and the natural transformation $\Gamma_{(r)}\rightarrow \id$
from~\ref{seq of nat transformations} is naturally isomorphic to
the inclusion of that submodule.
\item
Finally, $H^n_{(r)}(M) \cong 0$ for all $n\geq 2$.
\end{itemize}
\end{prop}

\bibliography{/home/asalch/texmf/tex/salch}{}
\bibliographystyle{plain}
\end{document}